\documentclass[a4paper]{amsart}
\usepackage[utf8]{inputenc}
\usepackage{amsmath,amsthm,amssymb}
\usepackage{latexsym}
\usepackage{geometry}
\usepackage{xcolor}

\newcommand {\E} {{\mathbb E}}
\newcommand {\Proba} {{\mathbb P}}

\newcommand{\be}{\begin{equation}}
\newcommand{\ee}{\end{equation}}
\newcommand{\benn}{\begin{equation*}}
\newcommand{\eenn}{\end{equation*}}
\newcommand{\bea}{\begin{eqnarray}}
\newcommand{\eea}{\end{eqnarray}}
\newcommand{\beann}{\begin{eqnarray*}}
\newcommand{\eeann}{\end{eqnarray*}}

\def\e{\ensuremath{\varepsilon}}

\newtheorem{theorem}{Theorem}[section]

\newtheorem{proposition}[theorem]{Proposition}
\newtheorem{lemma}[theorem]{Lemma}
\newtheorem{assumption}[theorem]{Assumption}

\numberwithin{equation}{section}

\def\txtd{{\textnormal{d}}}
\def\txte{{\textnormal{e}}}

\def\cD{{\mathcal{D}}}

\title{Travelling waves for discrete stochastic bistable equations}

\author{Carina Geldhauser}
\address{University of Sheffield, School of Mathematics and Statistics, Hicks Building, Hounsfield Road, Sheffield S3 7RH, United Kingdom.}
\email{c.geldhauser@sheffield.ac.uk}

\author{Christian Kuehn}
\address{Technical University of Munich, Department of Mathematics, Boltzmannstr.~3, 85748 Garching bei M\"unchen, Germany.}
\email{ckuehn@ma.tum.de}

\begin{document}

\begin{abstract}
Many physical, chemical and biological systems have an inherent discrete spatial 
structure that strongly influences their dynamical behaviour. Similar remarks apply to 
internal or external noise, as well as to nonlocal coupling. In this paper we study the 
combined effect of nonlocal spatial discretization and stochastic perturbations on 
travelling waves in the Nagumo equation, which is a prototypical model for bistable 
reaction-diffusion partial differential equations (PDEs). We prove that under suitable parameter 
conditions, various discrete-stochastic variants of the Nagumo equation have solutions, which 
stay close on long time scales to the classical monotone Nagumo front with high probability
if the noise level and spatial discretization are sufficiently small. 
\end{abstract}

\maketitle

\textbf{Keywords:} Nagumo equation, bistability, stochastic partial differential equation, 
lattice differential equation, travelling wave, noise, discretization, 
Allen-Cahn equation, Ginzburg-Landau equation, $\Phi^4$ model, Schl\"ogl equation.

%%%%%%%%%%%%%%%%%%%%%%%%%%%%%%%%%%%%%%%%%%%%%%%%%%%%%%%%%%%%%%%%%%%%%%%%%%%%%%%%%%%%%%%%%%%%%%%%%%%%%%%%%
%%%%%%%%%%%%%%%%%%%%%%%%%%%%%%%%%%%%%%%%%%%%%%%%%%%%%%%%%%%%%%%%%%%%%%%%%%%%%%%%%%%%%%%%%%%%%%%%%%%%%%%%%

\section{Introduction}
\label{sec:intro}

The Nagumo~\cite{Nagumo} partial differential equation (PDE) for $V=V(t,x)\in\mathbb{R}$ 
is given by 
\be
\label{eq:nagumo} \tag{{Nag$\mathbb{R}$}}
 \partial_t V \;  = \;  \nu \partial_x^2 V \;  + \; f(V), \qquad (t,x) \in [0,\infty) \times \mathbb{R},
\ee
where $f(V) = f(V;a) = V (1-V)(V-a)$, where $a\in(0,1/2)$ and $b,\nu>0$ are parameters.
The PDE~\eqref{eq:nagumo} 
is a prototypical model of bistability in the sense that $V\equiv 0$ and $V\equiv 1$ are locally asymptotically 
steady states, while $V\equiv a$ is unstable. For any $a \in (0,1/2)$ there exist 
travelling front solutions 
\[
V(t,x) = V^{{\textnormal{TW}}}(x - ct)=V^{{\textnormal{TW}}}(\zeta),\quad \zeta:=x-ct  
\]
connecting the two locally stable states, i.e., 
$V^{{\textnormal{TW}}}(-\infty) = 0$ and $V^{{\textnormal{TW}}}(\infty) = 1$. The front
is spatially monotone $(V^{{\textnormal{TW}}})'(\zeta)>0$, left-moving with a unique wave speed 
satisfying $c=c(a)<0$, unique up to translation, and (locally) nonlinearly stable~\cite{Evans,KuehnBook1}.
Extensions to the standing wave for $a=1/2$, and to right-moving waves for $a\in(1/2,1)$ 
are easily obtained from symmetry arguments~\cite{Evans,KuehnBook1}.
 
The Nagumo equation plays an important role in neuroscience~\cite{ErmentroutTerman} as the 
simplest toy model of signal propagation through axons. It is very actively studied also 
outside neuroscience applications, e.g., as an amplitude equation~\cite{CrossHohenberg}, in population 
dynamics modelling~\cite{CantrellCosner}, and in material science~\cite{AllenCahn}. 
In fact, the PDE~\eqref{eq:nagumo} is also referred to as the Allen-Cahn equation in 
materials science, as the real Ginzburg-Landau equation in normal form theory, as the 
$\Phi^4$-model in quantum field theory and as the Schl\"ogl model in chemistry. When 
modelling signal propagation in neurons, several effects are not taken into account 
in~\eqref{eq:nagumo}: 

\begin{itemize}
 \item[(I)] The electric signal travelling through a myelinated nerve fiber do not move 
continuously. The signal jumps from one gap in the myeline coating of the nerve fiber to the 
next~\cite{Keener}. This suggests the use of a \emph{spatially discrete} setting.
 \item[(II)] The propagation of the electric signal along the axon is influenced by many 
internal and/or external biophysical processes. Since modelling every process microscopically is
usually impossible, this leads naturally to a \emph{stochastic} version of the Nagumo equation.
 \item[(III)] The precise coupling distance of diffusion between myeline coating gaps 
is not easy to measure. This implies we should also allow for some form of \emph{nonlocal} coupling.  
 \item[(IV)] The axon does not have infinite length. Hence, one should consider \emph{bounded 
domains} instead. Furthermore, propagation takes place on a \emph{finite time scale}.
 \item[(V)] The propagation of fronts is an idealization of the electrical signal as usually
we would expect localized pulses. This requires \emph{systems} of reaction-diffusion equations.
\end{itemize}

Here we shall not cover the case (V), which is usually modelled using the 
Hodgkin-Huxley~\cite{HodgkinHuxley4} or FitzHugh-Nagumo~\cite{FitzHugh} PDEs but 
see~\cite{EichingerGnannKuehn}. However,
all the arguments we present can be carried over, in principle, to these cases. Instead,
we focus on a model to cover the \emph{combined} effects (I)-(IV). In fact, each of
the individual aspects (I)-(IV) have received some attention recently. We briefly 
review some background and introduce the relevant PDEs.\medskip   

The space-discrete setting will be modeled via a lattice differential equation (LDE), 
whose solution at node $i$, called $V_i=V_i(t)$, represents the potential at the $i$-th 
myeline gap. The discrete Nagumo equation with nonlocal diffusive coupling, reads at 
each node $i$ for some fixed coupling range $R\in \mathbb{N}$ as follows
\be
\label{eq:discnagumo}
\partial_t V_i = 
\;\frac{\nu}{R h^2}  \left(\sum_{j=-R}^R J(j) (V_j -  V_i) \right)+ \; f(V_i),\qquad i\in \mathbb{Z},  
\ee
where $h$ is a parameter controlling the discretization and $J(j) \in \mathbb{R}$
are weights. The classical case of local diffusive coupling is given by 
\be
\label{eq:discnagumolocal}
\partial_t V_i = 
\;\frac{\nu}{h^2}  \left(V_{i+1}-2V_i+V_{i-1}\right)+ \; f(V_i),\qquad i\in \mathbb{Z},  
\ee
The equation \eqref{eq:discnagumolocal} is the nearest-neighbor discretization 
of~\eqref{eq:nagumo}. The Nagumo LDE can be interpreted as being posed on an infinite 
lattice $\mathbb{Z}$ with lattice spacing $h$ so that $V_i$ corresponds to $V(ih)$. We 
write 
\begin{equation}
V^h:=(\ldots,V_{-2},V_{-1},V_0,V_1,V_2,\ldots)
\end{equation}
to emphasize that $V^h$ solves the discrete Nagumo equation.
The LDE~\eqref{eq:discnagumolocal} also admits travelling wave solutions for sufficiently 
strong diffusion strength $\nu$, i.e., for sufficiently large coupling; for small 
coupling, propagation failure may occur~\cite{Keener,HupkesSandstede2,Mallet-Paret}. 
More generally, the type of the discrete model may have substantial impact on the 
existence and uniqueness of travelling waves of the Nagumo and FitzHugh-Nagumo 
PDEs~\cite{ElmerVanVleck1,ElmerVanVleck2,HupkesvanVleck} as well as on the numerical 
analysis of discretization schemes for travelling waves~\cite{GriffithsStuartYee}. 

Notice that in~\eqref{eq:discnagumo} the general difference 
stencil involves $2R$ nodes, and $R$ may diverge with $N$, 
so it can be viewed as nonlocal. In fact, nonlocal variants
of the Nagumo equation have been studied in the LDE/PDE setting 
in several analytical and numerical works; see 
e.g.~\cite{AchleitnerKuehn2,BatesChenChmaj,BatesFifeRenWang,Chen1,ElmerVanVleck1}
and references therein. A similar difference stencil as used here was studied 
in~\cite{BatesChenChmaj,HupkesSchouten}, where existence of 
travelling wave solutions was proven for unbalanced nonlinearities and under certain 
conditions on the weights. 

Another important variation of the Nagumo equation is the stochastic PDE (SPDE)
version for $U=U(t,x)$ given by
\begin{equation}
\label{eq:nagumoSPDE}\tag{{SNag$\mathbb{R}$}}
 \partial_t U  = \nu \partial_{x}^2 U + f(U)+g(U)\xi,\qquad (t,x) \in [0,\infty) \times \mathbb{R}
\end{equation}
where $\xi=\xi(t,x)=\partial_t W(t,x)$ is a space-time dependent stochastic process, $W$ is a 
trace class Wiener process, and $g$ arises as a suitable mapping from modelling considerations, 
see \eqref{eq:sigmagrowth}-\eqref{eq:gzero}.

Although there is a detailed existence theory for many SPDEs~\cite{DaPratoZabczyk,Chow,LiuRoeckner1} 
going back to at least the late 1970s, and good physical understanding of many noisy pattern 
phenomena going back at least to the 1990s~\cite{GarciaOjalvoSancho}, the rigorous mathematical study 
of noisy (Nagumo) waves has just started to develop recently; see 
e.g.~\cite{HamsterHupkes,InglisMacLaurin,Lang,Stannat}. These studies have been driven 
by numerical observations~\cite{LordThuemmler,Shardlow,Tuckwell1} revealing that 
travelling wave solutions may persist under stochastic forcing, but their speed and 
form may change with varying noise strength. Of course, these results are also connected 
to recent advances in the numerical analysis of classical numerical schemes for the Nagumo
SPDE~\cite{SauerStannat1}; see also e.g.~\cite{Gyongy,Jentzen,LordRougement}. Furthermore, we 
refer to a recent survey of stochastic travelling wave problems for scalar reaction-diffusion 
equations SPDE for additional detailed background references~\cite{KuehnSPDEwaves}.\medskip

In this paper we are interested in the combined influence of (I)-(IV) on the 
finite-time evolution of travelling fronts. In this context, the key object is 
we are going to study is the stochastic LDE (SLDE) 
 \begin{equation}
\label{eq:discstochnagumointro}\tag{{dSNag$N$}}
   \txtd u_i = \left(\frac{\nu}{R h^2} \left(\sum_{j=-R}^R J(j) (u_j -  u_i) 
	\right)+f(u_i)\right)~\txtd t + g_N(u_i)~\txtd B_i,
 \end{equation}
with $i\in\{1,2,\ldots,N\}$, $u_i=u_i(t)$ stochastic processes on the lattice points $i$, 
independent identically distributed (iid) Brownian Motions $B_i(t)$, a suitable matrix-valued map  
$g_N$ obtained as a projection of $g$, and $R \leq N $ with $R\in\mathbb{N}$. In addition to 
viewing the solution 
as a vector 
%\begin{equation}
$u^h:=(u_1,u_2,\ldots,u_N)$
%\end{equation}
we may interpret the solution $u^h$, say via piecewise 
linear interpolation, as a function on the interval $\mathcal{D}:=[-L,L]$ with $u_1$ and 
$u_N$ corresponding to the values at the left and right endpoints. Despite its
evident importance for applications, particularly in the context of neuroscience, 
there seems to be no study available regarding the \emph{dynamics} 
of~\eqref{eq:discstochnagumointro} although some first study without reference
to dynamics is~\cite{BovierGeldhauser}. One potential reason could be that 
physical intuition would lead us to believe that the effects coming from (I)-(IV) 
are somehow ``small'' so that we can neglect them, at least in certain parameter regimes. 
To make this intuition 
mathematically precise is a key contribution of our study. For parameters for which 
travelling waves to the deterministic PDE are known to exist, we prove in the 
stochastic setting that for sufficiently small $h$ and sufficiently small noise
that the solution to the Nagumo SLDE~\eqref{eq:discstochnagumointro} 
is close to a phase-adapted travelling front solution of the Nagumo 
PDE~\eqref{eq:nagumo} over finite time scales. Our main result can be  
stated as follows:

\begin{theorem}
\label{thm:intro}
Let $V^{\textnormal{TW}}=V^{\textnormal{TW}}(t,x)$ be a travelling front solution 
to~\eqref{eq:nagumo}, $u^h_0$ be deterministic lattice initial data and $u^h$ a 
solution to ~\eqref{eq:discstochnagumointro} on the interval $\mathcal{D}:=[-L,L]$ to 
the initial data $u^h_0 \equiv u^h(0)$. Suppose $L >0$ is large enough, while 
$\delta>0$, $T>0$, and $\tilde{\varepsilon}>0$ are given. Suppose the initial data 
$u^h(0)$ satisfies
\be
\| u^h(0) - V^{{\textnormal{TW}}}(0,\cdot)\|^2_{L^2(\mathbb{R})} < \varepsilon .
\ee
Then there exists $\varepsilon>0$ and $c\in\mathbb{R}$ such that, for sufficiently 
small noise and sufficiently small $h$, we have for the 
solution $u^h=u^h(t)$ of~\eqref{eq:discstochnagumointro} the estimate
\be
 \Proba \left[ \sup_{t \in [0,T]} \| u^h(t) - V^{{\textnormal{TW}}}(\cdot-ct) 
\|_{L^2(\mathbb{R})} > \delta \right] \; \leq \;  \tilde{\varepsilon}.
\ee
\end{theorem}

The precise formulation of the noise structure, such as the statement of ``sufficiently small noise'' 
will be discussed below, it mainly deals with a sufficiently small covariance 
of the underlying Wiener process, and the growth of $g$. In summary, 
Theorem~\ref{thm:intro} confirms our intuition from biophysics/neuroscience, i.e., 
the wave propagation mechanism is robust against structural nonlocal and stochastic 
perturbations, which makes the Nagumo equation a good model.\medskip 

Our proof relies on a discrete version of the  monotone operator theory approach 
to SPDEs, as presented in \cite{Pardoux} and described in the 
monographs~\cite{Chow,LiuRoeckner1}. Our proof essentially decomposes the different 
error terms~\cite{KuehnKuerschner}, e.g., the dynamical stochastic approximation error 
is treated separately from the discretization error in the stencil. Therefore, it is natural 
to consider several intermediate evolution equations, e.g., the Nagumo PDE on a bounded domain
\begin{equation}
\label{eq:nagumoPDED}\tag{{Nag$\mathcal{D}$}}
 \partial_t v  = \nu \partial_{x}^2 v + f(v),\qquad (t,x) \in [0,T) \times \mathcal{D},
\end{equation}
with Neumann boundary conditions, and similarly the Nagumo SPDE on a bounded domain
\begin{equation}
\label{eq:nagumoSPDED}\tag{{SNag$\mathcal{D}$}}
 \partial_t u  = \nu \partial_{x}^2 u + f(u)+g(u)\xi,\qquad (t,x) \in [0,T) \times \mathcal{D}
\end{equation}
with suitable multiplicative noise desribed in section \ref{ssec:stochbasics}.

Hopefully, our notation conventions are by now already evident to the reader but let us stress
again that we use $v,V$ for the deterministic PDE solutions whereas $u,U$ are SPDE solutions. 
Small letter solutions $u,v$ are based on the bounded domain $\mathcal{D}$ and capital letter solutions
$U,V$ on the unbounded domain $\mathbb{R}$. Furthermore, discrete solutions will be treated as 
vectors $u^h,U^h,v^h,V^h$ or indicated by subindices. 

%%%%%%%%%%%%%%%%%%%%%%%%%%%%%%%%%%%%%%%%%%%%%%%%%%%%%%%%%%%%%%%%%%%%%%%%%%%%%%%%%%%%%%%%%%%%%%%%%%%%%%%%%
%%%%%%%%%%%%%%%%%%%%%%%%%%%%%%%%%%%%%%%%%%%%%%%%%%%%%%%%%%%%%%%%%%%%%%%%%%%%%%%%%%%%%%%%%%%%%%%%%%%%%%%%%
\section{Notation and Setting}
\label{sec:notation}

We discretize the interval $\mathcal{D} \subset \mathbb{R}$ into $N$ intervals of size $h$ 
and enumerate the respective grid points with the index $i$. The set of grid points is denoted 
by $\mathcal{D}^h$. We work with the Gelfand triple of Banach spaces 
\benn
H_0^1(\mathcal{D}) \cong W_0^{1,2}(\mathcal{D}) \subset L^2(\mathcal{D}) \subset H^{-1}(\mathcal{D}).
\eenn
Note that functions in the Sobolev spaces such as $L^2(\mathcal{D})$ or $W_0^{1,2}(\mathcal{D})$ 
evaluated on the grid $\mathcal{D}^h$ are $N$-dimensional vectors. Extending these functions in a 
piecewise linear manner, we can work with them also in the original Sobolev spaces. We choose an 
orthonormal basis $\{e_k\} \subset L^2(\mathcal{D})$, consisting of elements in $W_0^{1,2}(\mathcal{D})$, 
and span $\mathbb{R}^N$ with the first $N$ of these basis vectors. The projections of the Sobolev spaces 
on their first $N$ basis vectors can then be identified with $\mathbb{R}^N$, e.g.~$P_N H_0^1(\mathcal{D}) 
\cong \mathbb{R}^{N}$. To simplify notation, we denote both the scalar product on $L^2(\mathcal{D})$ and 
on $\mathbb{R}^N \cong P_N L^2(\mathcal{D})$ by $\big(\cdot , \cdot \big) $, while we denote the associated 
norm by $\| \cdot \|$. By $\langle \cdot , \cdot \rangle := \langle \cdot , \cdot 
\rangle_{H^{-1}(\mathcal{D}) , H_0^1(\mathcal{D})} $ we denote the dual product as well as the scalar 
product on $\mathbb{R}^N \cong P_N H_0^1(\mathcal{D}) $, where the projection to $\mathbb{R}^N$ is spanned by 
the first $N$ basis vectors of $H_0^1(\mathcal{D})$. Using the representation 
\begin{equation}\label{eq:onb}
 w = \sum_{k=1}^N \langle w, e_k\rangle e_k \qquad \textup{ for all elements } w \in H^{-1}(\mathcal{D}^h),
\end{equation}
we can work with the same basis vectors also in the space $H_0^1(\mathcal{D})$ and its dual.

\subsection{Operations with discrete-in-space functions}

Let $u^h(\cdot, t)$ be a piecewise linear function on the grid $\mathcal{D}^h$. There are 
several ways to define a (discrete) gradient. Using only two nodal values, we can identify 
$\nabla^h u^h(ih, t)$ either with the backward difference $D^-u_i(t) = 
\frac{1}{h}\left( u_i(t) - u_{i-1}(t)\right)$ or its adjoint, the forward difference 
$D^+u_i(t) = \frac{1}{h}\left( u_{i+1}(t) - u_{i}(t)\right)$. This choice leads to the 
discrete nearest-neighbour Laplacian as $\Delta^h u_i = D^+D^-v_i := 
h^{-2}\left(u_{i+1}  - 2 u_i + u_{i-1} \right)$.

We would like to use more general discrete stencils, which involve up to $R$ neighbours 
of $u_i$ in each direction, in other words involving the nodal values $u_{i-R} \ldots u_{i+R}$. 
We introduce coefficients $J(j)\in \mathbb{R}$ to attribute a weight of the $j$-th right 
neighbouring nodal value $u_{i+j}$. Such a  general second-order stencil then reads 
\begin{equation}
\label{eq:secondorder}
\begin{aligned}
\Delta_R^h u_i 
&= \frac{1}{h^2} \sum_{j=-R}^R J(j) (u_j -  u_i)  \\
&=  \frac{1}{h^2} \left( - 2 \sum_{k=1}^R J(k) u_i \; + \; \sum_{k=1}^R J(k) 
u_{i+k}  \; + \; \sum_{k=1}^R J(k) u_{i-k}   \right) \\
&=  \frac{1}{h^2}  \sum_{k=1}^R J(k) \left( - 2u_i \; + \; u_{i+k}  \; + \;  
u_{i-k}   \right) \\
&=  \frac{1}{h} \left( \frac{1}{h} \sum_{k=1}^R J(k)  (u_{i+k} - u_i)  \; - \; 
 \frac{1}{h}\sum_{k=1}^R J(k)  (u_i -  u_{i-k} )  \right). \\
\end{aligned}
\end{equation}
Note that $J(j)$ are fixed numbers and do not change with time, therefore the central difference 
operator $\Delta_R^h$ it is deterministic and time-independent.
It is therefore natural to define
 \begin{equation}
\begin{aligned}
\nabla^-_R  u_i  &:= (\Delta_R^h)^{1/2} u_i = \frac{1}{h} \sum_{k=1}^R J(k)  (u_i -  u_{i-k} ) 
\end{aligned}
\end{equation}
as the long-range analogues of the difference operators $D^-$. Using the adjoint operator 
to $\nabla^-_R$, denoted by $\nabla^+_R $, we can write \eqref{eq:secondorder} as
\begin{equation}\label{eq:secondorder2}
\begin{aligned}
\Delta_R^h u_i 
&= \nabla^+_R (\nabla^-_R  u_i).
\end{aligned}
\end{equation}
We need to impose conditions on the coefficients $J(j)$ to ensure 
that~\eqref{eq:secondorder} approximates a Laplacian.
To this aim, notice first of all that by construction 
$J(0) = - \sum_{j=-R, j \neq 0}^R J(j)$, which is a special case of diagonal dominance, 
from which we immediately conclude $\langle \Delta^h_R u^h , u^h \rangle \leq 0$. We
make the following assumptions throughout:
 
\begin{assumption}
\label{bed:J}
 The weights $J(j)\in \mathbb{R}$ satisfy
\begin{enumerate}
 \item[(A1)] $J(j) = J(-j)$
 \item[(A2)] $\sum_{j=-R}^R J(j) j^2 = 1$
 \item[(A3)] $\sum_{j=-R}^R J(j) j^4 < \infty \; $ or at least $ \; h^4 
\sum_{j=-R}^R J(j) j^4 \sim o(h^3) $
\end{enumerate}
\end{assumption}

The symmetry condition (A1) ensures that the operator $\Delta^h_R$ is self-adjoint 
in $\ell^2$. (A1) is often not strictly necessary mathematically, but it simplifies 
computations and is moreover very natural considering the real-world phenomena from 
which the model was derived. The moment conditions (A2)-(A3) guarantee that we approximate 
a Laplacian, as can be seen from the construction of finite difference operators 
via Taylor approximation at nodal distance $jh$, which gives 
\begin{equation}
 \sum_{j=-R}^R J(j) (v_j -  v_i) \; = \; h^2 \Big(\sum_{j=-R}^R j^2 J(j)\Big) 
\partial_{x}^2 v_i \; + \; \frac{h^4}{12} \Big(\sum_{j=-R}^R j^4 J(j)\Big) 
\partial_{x}^4 v_i \; + \; O(h^6).
\end{equation}
Note that, in contrast to \cite{BovierGeldhauser}, the assumptions on the moments imply 
a certain decay in the coefficients $J(j)$. This is because we allow for arbitrarily 
diverging stencil range $R$, especially also for $R = N$, while in the semigroup approach 
used in \cite{BovierGeldhauser}, the range was limited to $R \sim N^{1/2}$. 

\subsection{The probabilistic setting}\label{ssec:stochbasics}

Denote by $(\Omega,  \mathcal{F}, (\mathcal{F}_t)_t, \mathbb{P})$ a filtered probability 
space. A function $u: \mathcal{D} \times [0,T] \times \Omega \to \mathbb{R}$, which 
is evaluated on the grid $\mathcal{D}^h$, will be denoted by $u^h(\cdot , t, \omega)$ and 
at each node identified with a stochastic process $X^i_t(\omega)$, which takes values in 
$\mathbb{R}$.

We denote by $W: [0,T]\times \Omega  \to  L^2(\mathcal{D})$ a $Q$-Wiener process 
with values in $L^2(\mathcal{D})$. We assume that $W(t)=W(t,\omega)$ is adapted to the filtration 
$\mathcal{F}_t$.  We construct the covariance operator $Q$ such that is linear, bounded, self-adjoint, and positive semi-definite and that $Q$ has a common set of 
eigenfunctions with $\Delta$, so that
\be
\label{eq:ewtrace}
 Q e_k = \mu_k e_k. 
\ee
Moreover we ensure that $Q$ is of trace class, i.e., $ M_Q := {\rm Tr}\, Q < +\infty$, which implies 
that the sum of the eigenvalues of $Q$ is bounded $ \sum_{k=1}^{\infty} \mu_k < \infty$. 
It is well known that a $Q$-Wiener process in $L^2(\mathcal{D})$ can be represented in 
$L^2(\Omega , C([0,T], L^2(\mathcal{D})))$ using a sequence of iid Brownian motions 
$\left\{ B_j \right\}_{j \in \mathbb{N}}$ and considering the series
\be
\label{eq:tracenoise}
W(t)=W(\cdot,t) =  \sum_{k=1}^{\infty} \sqrt{\mu_k}  e_k(\cdot) B_k(t).
\ee
By means of an exponential inequality and the Borel-Cantelli Lemma, the convergence of the 
series can be obtained uniformly with probability one. Thus, the sample paths of $W(t)$ 
belong to $C([0,T], L^2(\mathcal{D}))$ almost surely, and we may therefore choose a 
continuous version. 

\subsection{The stochastic Nagumo equation} 
\label{subsec:properties}

The stochastic Nagumo equation we are using in this work is a perturbation 
of the (deterministic) Nagumo equation \eqref{eq:nagumo}.  
As stochastic perturbation, we choose a $Q$-Wiener Process $W(t)$ on 
$L^2(\mathcal{D})$ with covariance operator $Q$ being positive semi-definite, 
symmetric and of trace class. Moreover, we take a multiplicative noise term 
called $G(u): L^2(\mathcal{D}) \to \mathcal{H}$,  where we denote by 
$ \mathcal{H}$ the space of Hilbert-Schmidt operators, and assume it is 
Lipschitz continuous and satisfies linear growth conditions. More precisely, 
we assume
\be\label{eq:sigmagrowth}
 \| G (u) \|_{\mathcal{H}}^2 \; \leq \; c (1 + \| u\|^2) \qquad 
\textup{ a.e. } (t, \omega) \in [0,T] \times \Omega
\ee
and 
\be\label{eq:sigmalip}
 \| G (u) - G (v) \|_{\mathcal{H}}^2 \; \leq \; c (\| u - v \|^2) 
\qquad \textup{ a.e. } (t, \omega) \in [0,T] \times \Omega 
\ee
for all $u,v \in L^2(\mathcal{D})$.

For the rest of this work, we will avoid the operator notation and use the representation  
$\left( G(u) \chi \right) (x) \;  := \; g(u(x))\chi(x)$ for $u, \chi \in L^2(\mathcal{D})$ and  $g: \mathbb{R} \to \mathbb{R}$. 
In this notation, we state the core modelling  assumption   
\be\label{eq:gzero}
 g(0) = g(1) = 0,
\ee
which means that the effect of the noise should be concentrated on the wave front.
It is well-known that for such noise, if the initial data $u_0(x) \in [0, 1]$ for all $x \in  \mathcal{D}$, 
the solution also satisfies $u_0(x) \in [0, 1]$ for all $x \in \mathcal{D}$ and $t > 0$. 

To summarize, the lattice equation \ref{eq:discstochnagumointro}
we are considering in this work should be an approximation of a stochastic Nagumo equation 
on a bounded interval $[-L,L]$ driven by multiplicative trace-class noise, which acts only 
on the front. The continuum model is given by
\begin{equation}
\label{eq:stochnagumo}\tag{{SNag$\mathcal{D}$}}
  \txtd u(t) = [ \nu \partial_{x}^2 u(t) + f(u(t))] \txtd t + g(u(t))~\txtd 
	W(t) \qquad \qquad \textup{ on } \mathcal{D} \times [0,T].
\end{equation}
The existence of mild solutions to \eqref{eq:stochnagumo} for Lipschitz 
nonlinearities is classical, see e.g.~\cite{DaPratoZabczyk,KuehnNeamtu2}. Using a 
localization and truncation argument, see e.g.~\cite{cerrai1999, BovierGeldhauser}, 
local-in-time results can be carried over polynomial nonlinearity $f$ 
with one-sided Lipschitz condition such as in~\eqref{eq:stochnagumo}, while
global-in-time results have to exploit the sign in the cubic nonlinearity 
leading to dissipativity for large $|u|$~\cite{BerglundGentz10,Robinson1,KuehnNeamtuPein}. 
Via monotone operator theory, one may see furthermore \cite{Pardoux} that~\eqref{eq:stochnagumo} 
admits a variational solution in $L^2(\Omega , C([0,T],L^2(\cD))) \cap 
L^2(\Omega \times [0,T], H_0^1(\cD))$. In particular, we have that almost 
surely $u \in L^{\infty}([0,T];L^2(\cD)) \cap L^2([0,T]; H_0^1(\cD))$. 
Due to It\^{o}'s formula, the stochastic Nagumo equation satisfies an energy 
equation of the form
\begin{equation}
\label{eq:continuousito}
\begin{aligned}
\E \left[\|u(t)\|^2_{L^2} \right]  & \; = \; \|u(0)\|^2_{L^2} + \; 2 \nu 
\E \left[ \int_0^t \left( \Delta u (s), u(s)\right) ~\txtd s \right]  \\ 
  & + \, 2  \; \E \left[ \int_0^t \left( f(u(s)), u(s)\right)~\txtd s\right]
		\;  + \E \left[ \int_0^t g(u(s))^2~ \txtd s \right] .
\end{aligned}
\end{equation}
Regarding the stochastic LDE version of the Nagumo SPDE, we may introduce the convenient abbreviation 
\benn
\Delta_R^h u_i := \frac{1}{R h^2} \sum_{j=-R}^R J(j) (u_j -  u_i)
\eenn
to indicate that the difference stencil can be regarded a (generalized) discretization of the Laplace 
operator, if the coefficients $J(j) \in \mathbb{R}$ satisfy Assumption \ref{bed:J}. This fact will be 
justified in more detail in the next section. Moreover, taking advantage of the vector notation 
$u^h:=(u_1,u_2,\ldots,u_N)$, we may write the discrete-in-space evolution \eqref{eq:discstochnagumointro} as 
\be
\label{eq:discstochnagumo}\tag{{dSNag}}
  \txtd u^h (t)  =  \nu \Delta_R^h u^h (t)
  \txtd t  + f( u^h(t) ) ~\txtd t \; + \; g(u^h(t))~ \txtd W^h(t)  \qquad \textup{on } \mathcal{D}^h \times [0,T],
\ee
where we denoted by $W^h (t)$ a sufficiently large, yet finite, partial sum to the infinite sum 
in~\eqref{eq:tracenoise} and denoted $\mathcal{D}^h := \lbrace 1, \ldots, N\rbrace$. Recalling $\left( G(u) \chi \right) (x) \;  := \; g(u(x))\chi(x)$ for $u, \chi \in L^2(\mathcal{D})$, and that $g$ is a function $g:\mathbb{R} \to \mathbb{R}$, we may define the discrete 
multiplicative noise operator in the same way for $u^h$ (without changing the notation). It then obviously satisfies~\eqref{eq:sigmagrowth} 
and~\eqref{eq:sigmalip}. Due to the trace class assumption, we may always select the truncation level 
for the Wiener process sufficiently large to guarantee that solutions of~\eqref{eq:discstochnagumo} stay 
close to the same equation driven by $W$.

\subsection{Monotone operators}

The following paragraph recalls that, due to the properties of $f$, the sum 
$\nu \Delta + f$ defines a monotone operator. In the continuous context, this 
is well-understood: a concise treatment of the theory of monotone operators can 
be found for example, in \cite{Zeidler2b}, or, including local monotonicity, 
in \cite{LiuRoeckner1}. We will briefly state those precise properties, which 
will be used in the proofs. First, we note that the nonlinear term $f(u)$ is 
Lipschitz continuous w.r.t $u$ on bounded subsets of $H_0^1(\mathcal{D})$ with 
Lipschitz constant independent of $t$. More precisely, we have the following 
standard results, which we include for completeness here:

\begin{lemma}
\label{lemma:liploc}
For any $M >0$, there exists a constant $K_M > 0$ such that the local 
Lipschitz continuity condition holds:
 \be
 \label{eq:liploc}
  \| f(v_1) - f(v_2) \|^2 \; \leq \; K_M \| v_1 - v_2 \|_{H_0^1(\mathcal{D})}^2 
	\qquad \textup{a.e. } (t,\omega) \in [0,T] \times \Omega
 \ee
for any $v_1 , v_2 \in H_0^1(\mathcal{D})$ with $\| v_1 \|_{H_0^1(\mathcal{D})}^2 
< M$ and $\| v_2 \|_{H_0^1(\mathcal{D})}^2 < M$ 
\end{lemma}

\begin{proof}
We have 
 \be
  \| v_1^3 - v_2^3 \|^2 = \| (v_1^2 + v_1v_2 + v_2^2)(v_1 - v_2) 
	\|^2 \leq \; 8 \; \left( \|v_1^2 (v_1 -v_2)\|^2 \; + \; 
	\|v_2^2 (v_1 - v_2)\|^2 \right)
 \ee
By Sobolev embedding, we can get for some constants $C_1, C_2 >0$ the 
estimates $\|v_1\|_{L^4} \leq C_1 \|v_1\|_{H_0^1}$ and $\|v_1^2v_2\|_{L^2}^2 
\leq C_2 \|v_1\|_{H_0^1}^4 \|v_2\|_{H_0^1}^2$. 
Hence, there exists a constant $C_3>0$ such that 
\be
   \| v_1^3 - v_2^3 \|^2 \leq \; C_3 \; \left( \|v_1\|_{H_0^1}^4 \; + \;
		\|v_2\|_{H_0^1}^4  \right) \| v_1 - v_2 \|_{H_0^1}^2
\ee
which satisfies \eqref{eq:liploc}.
\end{proof}

Furthermore, we can derive estimates similar to Lemma \ref{lemma:liploc}, providing us 
growth bounds such as
\be
 \|f(v)\|_{H^{-1}(\mathcal{D})} \leq c_1 \|v\|_{H_0^1(\mathcal{D})} 
\left( 1 + \|v|^2_{L^2(\mathcal{D})} \right)
\ee
as well as
\be
 \|f(v_1) - f(v_2)\|_{H^{-1}(\mathcal{D})} \leq c_2  \left(1 + 
\|v_1\|^2_{H_0^1(\mathcal{D})} + \|v_2\|^2_{H_0^1(\mathcal{D})} 
\right) \| v_1 - v_2 \|_{L^2(\mathcal{D})}.
\ee
Moreover, the combined operator $A := \nu \Delta + f$ is obviously hemi-continuous 
in $H_0^1(\mathcal{D})$, in the sense that for all 
$v_1, v_2, v_3 \in H_0^1(\mathcal{D})$ and $t \in [0,T]$ the mapping
\be
 \theta \mapsto \langle A(v_1 + \theta v_2) , v_3 \rangle 
\ee
is continuous from $\mathbb{R}$ into $\mathbb{R}$. Due to \eqref{eq:sigmagrowth} and 
\eqref{eq:sigmalip}, we know~\cite{Pardoux,Chow,LiuRoeckner1,Stannat} that the 
sum of operators satisfies for all $v \in H_0^1(\mathcal{D})$, $ t\in [0,T]$ 
the coercivity condition
\be
\label{eq:coercivity}
\left\langle \nu \Delta v + f(v) , v \right\rangle \, + \| g(v)\|_{\mathcal{H}} 
\leq  \; - \, \nu \| v\|_{H_0^1(\mathcal{D})}^2+(c_a+\nu)\| v\|_{L^2(\mathcal{D})}^2,
\ee
with $ c_a = \sup_{\upsilon \in \mathbb{R}}  f'(\upsilon)  =  \frac{1}{3} 
\big( a^2 - a + 1\big)$, holds. Finally, the sum of operators satisfies for all 
$v_1, v_2 \in H_0^1(\mathcal{D})$ the  monotonicity condition
\be
\label{eq:monotonicity}
\left\langle \nu \Delta v_1 + f(v_1) - \nu \Delta v_2 - f(v_2), v_1 - v_2 
\right\rangle \, + \| g(v_1) - g(v_2)\|_{\mathcal{H}}^2 \leq  \, c_a \| v_1 - v_2 \|^2
\ee
on $[0,T] \times \Omega$. We will now verify similar properties hold in the 
discrete setting of our LDE \eqref{eq:discstochnagumo}. The proof is elementary, yet 
we will provide it in detail to make the strategy transparent for stochastic LDEs.

\begin{lemma}
Let the conditions of Assumption \ref{bed:J} on the general stencil $ \Delta_R^h$ 
be satisfied. Then the discrete operators appearing in \eqref{eq:discstochnagumo} 
satisfy the following estimates:
\begin{itemize}
 \item[(L1)] coercivity
  \be
	\label{eq:discretecoercivitynoise}
  \sum_{i=1}^N  \left( \nu  \Delta_R^h u_i + f(u_i) \right) u_i \, + \| g(u_i)\|^2 
	\; \leq  \; - \, \nu \| \nabla^-_R  u^h\|^2 + c_a \| u^h \|^2
 \ee
 \item[(L2)] monotonicity 
 \be
 \label{eq:discretemonotonicitynoise}
 \sum_{i=1}^N  \left( \nu \Delta_R^h u_i - \nu  \Delta_R^h v_i + f(u_i) - f(v_i)\right) 
(u_i - v_i )\, + \|g(u_i) - g(v_i)\|^2 \, \leq  \, c_a \| u^h - v^h \|^2.
 \ee
\end{itemize}
\end{lemma}

\begin{proof}
First, note that
\begin{equation}\label{eq:pre-coercivity}
 \sum_{i=1}^N f(u_i) \cdot u_i =  \sum_{i=1}^N \left( f(u_i) - f(0) \right)
\left( u_i - 0 \right) =  \sum_{i=1}^N \left( \frac{f(u_i) - f(0)}{u_i - 0} \right)
\left( u_i - 0 \right)^2 \leq c_a  \sum_{i=1}^N u_i^2 
\end{equation}
where we used the mean-value theorem in the last inequality. 
Second, we look at the discrete integration by parts formula, which reads 
in its standard form
\begin{equation}
  \sum_{i=1}^N \Delta^h u_i \cdot u_i =   \sum_{i=1}^N D^+ (D^- u_i) \cdot u_i 
	= -  \sum_{i=1}^N (D^- u_i) \cdot D^- u_i = - \sum_{i=1}^N (D^- u_i)^2.
\end{equation}
It can be extended to the long-range case via the operators $\nabla^-_R $ 
and $\nabla^+_R$. Using the last equation and previous inequality, we can derive easily, 
in the special case of the nearest-neighbour stencil, the coercivity
\begin{equation}
\label{eq:discretecoercivitynoise_nn}
 \sum_{i=1}^N  \left( \nu  D^+ (D^- u_i) + f(u_i) \right) u_i \, +
 \| g(u_i)\|^2 \; \leq  \; - \, \nu \| D^- u^h\|^2 + c_a \| u^h \|^2
\end{equation}
where we used \eqref{eq:sigmagrowth}. Using \eqref{eq:sigmalip}, we get the 
monotonicity of the sum of operators 
\begin{equation}
\label{eq:discretemonotonicitynoise_nn}
\sum_{i=1}^N  \left( \nu  D^+ (D^- u_i) - \nu  D^+ (D^- v_i) + f(u_i) -
 f(v_i)  \right) (u_i - v_i )\, + \|g(u_i) - g(v_i)\|^2 \, \leq  \, c_a 
\| u^h - v^h \|^2.
\end{equation}
 For the general case, it was already noted (without proof) by Bates, Chen, 
Chmai \cite{BatesChenChmaj}, that general stencils of the form \eqref{eq:secondorder} 
satisfy the monotonicity condition with $c = 0$ instead of $c_a$. 
Indeed, testing with $u^h$ and using the summation by parts formula we obtain
\begin{equation}
\label{eq:longrangesummationbyparts}
\begin{aligned}
\langle \Delta_R^h u^h , u^h \rangle 
&= \frac{1}{h^2}  \sum_{i=1}^N \left(\sum_{j=-R}^R J(j) (u_j -  u_i) 
\right) \cdot u_i  \\
&= \sum_{i=1}^N \nabla^+_R (\nabla^-_R  u_i) \cdot u_i \\
&= -  \langle \nabla^-_R  u^h , \nabla^-_R  u^h \rangle  \; = \; - \| 
\nabla^-_R  u^h \|^2 \; \leq \; 0 .
\end{aligned}
\end{equation}
Repeating the strategy of the nearest neighbour case, using 
again \eqref{eq:pre-coercivity}, \eqref{eq:sigmagrowth}, \eqref{eq:sigmalip} 
and \eqref{eq:longrangesummationbyparts} gives us 
immediately~\eqref{eq:discretecoercivitynoise} 
and~\eqref{eq:discretemonotonicitynoise}, which concludes the proof. 
\end{proof}

%%%%%%%%%%%%%%%%%%%%%%

%%%%%  END OF NOTATION

%%%%%%%%%%%%%%%%%%%%%%

\section{Existence and stability of travelling wave fronts}

Recall that our goal is show that the stochastic LDE~\eqref{eq:discstochnagumo} admits, 
for sufficiently small $h$ and sufficiently small noise, travelling 
front-like solutions, in the sense that its solutions are very likely to 
be close to classical deterministic travelling fronts; see also Theorem~\ref{thm:intro}.

To approach this problem, we use several ingredients: first, results on the existence 
and properties of solutions to the LDE \eqref{eq:discstochnagumo}; second, the convergence 
of the solutions of the $h$ approximations to the solution of \eqref{eq:stochnagumo}; third, 
approximations of the classical deterministic front via the truncated spatial 
problem~\eqref{eq:nagumoPDED}; fourth, small noise stability of traveling wave fronts 
for \eqref{eq:nagumoSPDE}. 

The rest of this paper is organized along these ingredients: In Section~\ref{sec:apriori}, 
we investigate existence and properties of solutions to the LDE \eqref{eq:discstochnagumo}.
We follow this up by a discrete-to-continuum convergence result in Section~\ref{sec:convergence}. 
In Section~\ref{ssec:cutoff}, we study the truncation error of restricting the solutions to 
a bounded interval and in Section~\ref{ssec:smallnoise} we quantify the error coming from 
the stochastic perturbation. In Section~\ref{sec:maintheorem}, we finally obtain the main 
result, incorporating an SPDE small-noise stability result.
 
\subsection{A priori estimates}
\label{sec:apriori}

Recall again that we are going to employ the notation $\| \cdot \| := \| \cdot \|_{L^2(\mathcal{D})}$.
Rigorous results on the existence and properties of solutions to equation~\eqref{eq:discstochnagumo} 
are obtained in the classical framework of strong solutions of stochastic ordinary differential equations 
(SODEs). The key part of this section is the following a priori estimate:

\begin{proposition}
\label{prop:aprioriu}
Let the initial data $u_0\in C^4(\mathcal{D})$ be deterministic and suppose
Assumption~\ref{bed:J} is satisfied. 
Then the solution $u^h$ of~\eqref{eq:discstochnagumo} satisfies for any $h$  
\begin{equation}
 \E \left[ \int_0^T \|\nabla_R^- u^h(t)\|^2~\txtd t \; + \; 
\sup_{t \leq T} \|u^h(t)\|^2 \right] < \infty.
\end{equation}
\end{proposition}

The proof of Proposition \ref{prop:aprioriu} is split into three parts, which are 
Lemmata \ref{lemma:energyito}, \ref{lemma:boundedH1} and \ref{lemma:boundedL2}, which
are proved under the same assumptions as Proposition \ref{prop:aprioriu}.

\begin{lemma}
\label{lemma:energyito}
For any $h >0$, the solution $u^h$ of~\eqref{eq:discstochnagumo} exists and 
satisfies an energy equality
\begin{equation}
\label{eq:energyito}
\begin{aligned}
\|u^h(t)\|^2_{L^2} & \; =  \; \|u^h(0)\|^2_{L^2} \; + \; 2 \nu \int_0^t 
\left\langle \Delta_R^h u^h(s), u^h(s)\right\rangle~\txtd s \\ 
  & + \, 2 \; \int_0^t \left( f(u^h(s)), u^h(s)\right) ~\txtd s\;  
	+ \; 2 \int_0^t g(u^h(s)) \left( u^h(s), \txtd W^h(s)\right) +  \int_0^t g(u^h(s))^2~ 
	\txtd s
\end{aligned}
\end{equation}
\end{lemma}

\begin{proof}
The stochastic LDE~\eqref{eq:discstochnagumo} is just a system of SODEs, which is rigorously 
stated in integral form
\begin{equation}
\label{eq:discstochnagumointegrated}
\begin{aligned}
u_i(t) &= u_0 + \frac{\nu}{h^2}\int_0^t 
\sum_{j=-R}^R J(j) (u_j -  u_i) ~\txtd s\;  + \; \int_0^t f(u_i(s)) ~\txtd s 
+ \int_0^t g(u_i) ~\txtd W_i(s) \;, \quad i=1, \ldots, N.
\end{aligned}
\end{equation}
From \eqref{eq:discstochnagumointegrated} it is easily seen that for each single $i$, the 
stochastic LDE is in fact an It\^{o} equation, for which it is well known 
(see e.g.~\cite{KaratzasShreve}) that there exists a solution, which is an adapted process. 
Moreover, as $u_i$ are the coefficients of $u^h$ in the basis of $\mathbb{R}^N \cong 
P_N H_0^1(\mathcal{D})$, \eqref{eq:discstochnagumointegrated} is a finite-dimensional 
It\^{o} equation, which therefore has a solution as an adapted process. This adapted process 
also has a continuous version. To derive the energy equation, for some fixed 
$h = \frac{1}{N}$ we build a solution vector $u^h$ via the integral form of the 
stochastic LDE \eqref{eq:discstochnagumo} given by 
equation~\eqref{eq:discstochnagumointegrated}. We can write the solution, 
using an orthonormal basis $e_i$ of $V^h := P_N H_0^1$, in the following form
\begin{equation}
\label{eq:tested1}
\begin{aligned}
\left(u^h(t), e_i \right)  \, & = \; \left(u^h_0,e_i\right) \; + \; \nu 
\int_0^t \left\langle \Delta_R^h u^h, e_i \right\rangle ~\txtd s + \, \; \int_0^t 
\left( f(u^h) ,e_i \right)~\txtd s\; \\
& + \;  \int_0^t g (u^h) \Big(e_i, \txtd W^h(s)\Big). 
\end{aligned}
\end{equation}
Next, in \eqref{eq:tested1} we sum up from $i=1$ to $i=N$ and apply It\^{o}'s formula 
in finite dimensions to get
\begin{equation}
\label{eq:energyitofinal}
\begin{aligned}
\|u^h(t)\|^2 & \; =  \; \|u^h(0)\|^2 \; + \; 2 \nu \int_0^t \left\langle 
\Delta_R^h u^h(s), u^h(s)\right\rangle ~\txtd s 
  + \, 2\; \int_0^t \left( f(u^h(s)), u^h(s)\right) ~\txtd s\;\\  & \; + \; 2 \int_0^t g(u^h(s)) 
	\left( u^h(s), \txtd W^h(s)\right) \; + \;  \int_0^t g(u^h(s))^2~\txtd s
\end{aligned}
\end{equation}
which is exactly \eqref{eq:energyito}.
\end{proof}

With the way of writing \eqref{eq:tested1}  we already point to the fact that we consider 
the stochastic LDE as a generalized Galerkin approximation of the stochastic Nagumo 
equation. Indeed, recall that \eqref{eq:stochnagumo} has a solution whose trajectory 
is in $L^2([0,T], H_0^1(\mathcal{D}))$, $\Delta u(\cdot) \in L^2([0,T], H^{-1}(\mathcal{D}))$.
Hence, we may write the stochastic Nagumo equation in weak form, using the scalar product 
$( \cdot , \cdot)_{L^2(\mathcal{D})}=(\cdot,\cdot)$ and the dual product 
$\langle \cdot , \cdot \rangle$ as follows
\begin{equation}
\label{eq:testedgeneral}
\begin{aligned}
\left(u(t), \varphi \right)  \, & = \; \left(u^h_0,  \varphi  \right) \; + \; \nu \int_0^t 
\left\langle \Delta  u(s),  \varphi \right\rangle ~\txtd s \; + \; \int_0^t \left( f(u(s)) 
, \varphi  \right) ~\txtd s\; \\
& + \;  \int_0^t \left( g (u(s)) , \varphi \right) ~\txtd W(s). 
\end{aligned}
\end{equation}
Due to~\eqref{eq:tracenoise}, we can interpret the stochastic integral term as 
 \begin{equation}
\begin{aligned}
 \int_0^t \Big( g (u(s)) , \varphi \Big) \txtd W(s) 
= \sum_{k=1}^{\infty} \int_0^t \Big( g_k (u(s)) , 
\varphi \Big)~\txtd W_k(s) = \int_0^t \Big( \varphi , g (u(s))  ~\txtd W(s) \Big)
\end{aligned}
\end{equation}
and relate the notation of the stochastic integral term in~\eqref{eq:energyitofinal} 
to the weak form~\eqref{eq:testedgeneral}.

\begin{lemma}
\label{lemma:boundedH1}
The discrete solution $u^h$ 
of~\eqref{eq:discstochnagumo} is uniformly bounded in $L^2(\Omega \times [0,T], 
H_0^1(\mathcal{D}))$, i.e.
\begin{equation}
\sup_h \E \left[ \int_0^T \|\nabla_R^- u^h(t)\|^2_{L^2(\mathcal{D})}  \txtd t  \, \right] 
\; < \; \infty.
\end{equation}
\end{lemma}

\begin{proof}
We start with the energy equation \eqref{eq:energyito}. 
Taking the expectation, the stochastic integral is zero and we arrive at 
\begin{equation}
\label{eq:expectationstart}
\begin{aligned}
\E \left[\|u^h(t)\|^2 \right] -  \; \E \left[ \|u^h(0)\|^2 \right]  
&=   \; 2 \nu \E \left[ \int_0^t \left( \Delta_R^h u^h(s), u^h(s)\right) ~\txtd s \right]  \\ 
  & + \, 2 \; \E \left[ \int_0^t \left( f(u^h(s)), u^h(s)\right) ~\txtd s\right]  
	\;  + \E \left[ \int_0^t g(u^h(s))^2 ~\txtd s \right] 
\end{aligned}
\end{equation}
Abbreviate now the right hand side of \eqref{eq:expectationstart} by 
$RHS := 2 \nu \E \left[ \int_0^t \left( \Delta_R^h u^h(s), u^h(s)\right)~\txtd s \right]  
+  2\E \left[ \int_0^t \left( f(u^h(s)), u^h(s)\right) ~\txtd s\right]  
+ \E \left[ \int_0^t g(u^h(s))^2 ~\txtd s \right] $. We use the coercivity 
estimate~\eqref{eq:discretecoercivitynoise} to get that the right hand side 
of~\eqref{eq:expectationstart} satisfies
\begin{equation}
\label{eq:step11}
\begin{aligned}
RHS & \; \leq  - \; 2 \nu \E \left[  \int_0^t   \|\nabla_R^- u^h(s)\|^2 \, ~\txtd s \right]   
+ 2 (c_a + \nu) \E \left[  \int_0^t   \| u^h (s)\|^2 ~\txtd s \right] 
\end{aligned}
\end{equation}
which is, as the initial data is deterministic,
\begin{equation}
\label{eq.startgronwall}
\begin{aligned}
\E \left[ \|u^h(t)\|^2 \right]   +  2 \nu \E \left[  \int_0^t   \|\nabla_R^- u^h(s)\|^2 
\, \txtd s \right]  & \; \leq  \;  \|u_0\|^2 \\
&+ 2  (c_a + \nu) \E \left[  \int_0^t   \| u^h (s) \|^2 ~\txtd s \right].  
\end{aligned}
\end{equation}
Now we apply Gronwall's Lemma to $\E \left[ \|u^h(t)\|^2 \right]$ to get 
\begin{equation}
\label{eq.gronwallcalc1}
\E \left[ \|u^h(t)\|^2 \right] \; \leq \; \txte^{2  (c_a + \nu) t } \; \|u_0\|^2 .
\end{equation}
Furthermore, as the RHS is independent of $h$ for $t \in [0,T]$ we get 
\begin{equation}\label{eq.gronwallcalc2}
\sup_{0 \leq t \leq T} \E \left[ \|u^h(t)\|^2 \right] \leq c (a, \nu, T) \; \|u_0\|^2 \;  
\end{equation}
with a constant $c$ which is independent of $h$. Going back to~\eqref{eq.startgronwall}, 
we see that the term $\E \left[ \|u^h(t)\|^2 \right]$ on the left-hand side (LHS) is 
estimated against a constant by \eqref{eq.gronwallcalc2}, so we remain with the second 
term and get therefore its boundedness 
\begin{equation}
\label{eq.gronwallcalc3}
 \E \left[  \int_0^T   \|\nabla_R^- u^h(t)\|^2 \, \txtd t \,  \right] \; 
\leq  \; c (a, \nu, T, u_0) , 
\end{equation}
which is the desired estimate.
\end{proof}

Note that the constant in the Gronwall estimate grows exponentially with $t$, 
therefore \eqref{eq.gronwallcalc2} diverges for $T \to \infty$ but since we are
only going to work on finite time scales, this will not be relevant here. Moreover, 
we did not directly estimate the discrete gradient $ \|\nabla_R^- u^h(t)\|^2$, but 
we made use of the energy equation, which is a consequence of It\^{o}'s formula. 
Therefore, the exact range $R$ of the discrete stencil does not directly affect 
the a-priori estimates.

\begin{lemma}
\label{lemma:boundedL2}
The discrete solution $u^h$ of \eqref{eq:discstochnagumo} is 
$L^2(\Omega, C([0,T], L^2(\mathcal{D})))$ independently of $h$, i.e.,
\begin{equation*}
\E \left[ \; \sup_{t \leq T} \;  \|u^h(t)\|^2_{L^2(\mathcal{D})} \, \right]  
\; \leq \;  c (a, \nu, T, u_0)
\end{equation*}
\end{lemma}

\begin{proof}
We start with the energy equation \eqref{eq:energyito} over which we take the 
supremum in $t$ and the expectation  
\begin{equation*}
\begin{aligned}
\E \left[ \sup_{ t \leq T} \|u^h(t)\|^2 \right] & \; =  \; \|u^h(0)\|^2\; + 
\; 2 \nu \E \left[ \int_0^T \left( \Delta_R^h u^h(s), u^h(s)\right) ~\txtd s \right] \\ 
  & + \, 2 \;\E \left[  \int_0^T \left( f(u^h(s)), u^h(s)\right) ~\txtd s \right] \;  
	+ \; 2 \E \left[ \sup_{ t \leq T} \int_0^t g(u^h(s)) \left( u^h(s), \txtd W^h(s)\right) \right]\\
  & + \, \E \left[ \int_0^T g(u^h(s))^2 ~\txtd s \right]. 
\end{aligned}
\end{equation*}
The noise term can be analyzed using the Burkholder-Davis-Gundy inequality
\begin{equation*}
\begin{aligned}
\E \left[  \sup_{t \leq T} \left|\int_0^t g(u^h(s)) \left( u^h(s), \txtd W^h(s)\right) 
\right| \right] \;
& \leq \;  c \; \E \left[ \left( \int_0^{T} \left( g(u^h(s)) , 
u^h(s)\right)^2_{L^2(\mathcal{D})} \txtd s \right)^{1/2} \right] \\
& \leq \;  c \; \E \left[ \sup_{t \leq T} \| u^h(t) \|  \left( \int_0^{T}  
g(u^h(s))^2~\txtd s \right)^{1/2}    \right]\\
& \leq \;  \frac{1}{2} \E \left[ \sup_{t \leq T} \| u^h(t) \|^2  \right] \; 
+ \;  \frac{c^2}{2} \E \left[  \int_0^{T}  g(u^h(s))^2~\txtd s \right].
\end{aligned}
\end{equation*}
Estimating the other terms by coercivity, we get 
\begin{equation*}
\begin{aligned}
 \frac{1}{2} \E \left[ \sup_{t \leq T} \| u^h(t) \|^2  \right]  \; 
&\leq  \; \|u_0 \|^2 \; - \; 2 \nu \E \left[  \int_0^T   
\|\nabla_R^- u^h(s)\|^2 \, ~\txtd s \right] \; + \;c(c_a, \nu)  
\E \left[ \int_0^{T}   \| u^h(s) \|^2 ~\txtd s \right]. \;
\end{aligned}
\end{equation*}
Notice now that we can estimate, thanks to the last lemma, 
in particular \eqref{eq.gronwallcalc3},
\begin{equation}
 \E \left[ \int_0^{T}   \| u^h(s) \|^2 ~\txtd s \right] \; \leq \; 
\E \left[  \int_0^T   \|\nabla_R^- u^h(s)\|^2 \, ~\txtd s \right] 
\leq c (a, \nu, T)
\end{equation}
and as $ \|u_0\|^2 \leq c$ by assumption, so
\begin{equation*}
\E \left[ \sup_{t \leq T} \|u^h(t)\|^2_{L^2(\mathcal{D})} \right]  
\leq \; c (a, \nu, T, u_0)
\end{equation*}
which means that $u^h$ is bounded in $L^2(\Omega, C([0,T], L^2(\mathcal{D})))$ 
independently of $h$. 
\end{proof}

Note that we used here again the estimate \eqref{eq.gronwallcalc2}, which 
comes from Gronwall's inequality, so this result holds only for finite $t$.

\subsection{Convergence and identification of the limit}
\label{sec:convergence}

We begin with the proof of a simple lower semicontinuity statement on 
$\E\left[ \| u(T)\|^2_{L^2}\right]$, which we will use in the proof of 
the convergence theorem, precisely in equation \eqref{eq:uselschere}.
 
\begin{lemma}
\label{lemma:lscnorm}
Let $u\in L^2\left( \Omega \times [0,T]; H_0^1(\mathcal{D})\right) 
\cap L^2\left(\Omega; L^{\infty}([0,T],L^2(\mathcal{D})\right)$. For $u^h(t) 
\longrightarrow u(t)$ weakly in $L^2(\Omega, L^2(\mathcal{D}))$ it holds that
 \begin{equation}\label{eq:lsc1}
 \E \left[ \|u(T)\|^2_{L^2} - \|u_0\|^2_{L^2}\right] \; \leq \; 
\liminf_{ h \to 0}  \E \left[ \|u^h(T)\|^2_{L^2} - \|u^h_0\|^2_{L^2}\right]
\end{equation}
\end{lemma}
\begin{proof}
First, as $u^h  \rightharpoonup u$ in $L^2(\mathcal{D})$, the lower 
semicontinuity of the $L^2$-norm gives $ \| u(t)\|_{L^2} \leq \liminf_{h \to 0} 
\|u^h(t)\|_{L^2}$. As the mapping $u(t) \to \E\left[\|u(t)\|^2_{L^2}\right]$ 
is convex as a map from $L^2 (\Omega, L^2(\mathcal{D}))$ to $\mathbb{R}$, and by 
the same argument,  for any $t \in [0,T]$ also convex as a map from 
$L^2 (\Omega \times [0,T] ; L^2(\mathcal{D}))$ to $\mathbb{R}$, we get furthermore, 
\begin{equation}
\label{eq:utlsc}
\E\left[ \| u(T)\|^2_{L^2}\right] \; \leq \; \liminf_{h \to 0} 
\E\left[\|u^h(T)\|_{L^2}^2 \right].
\end{equation}
By strong convergence of the initial condition in $L^2(\mathcal{D})$, 
we have $u^h(0) = \sum_{i=1}^{\infty} (u_0, e_i) e_i \longrightarrow u_0$ and so
\begin{equation*}\label{eq:lsccalc}
\begin{aligned}
 \E \left[ \|u(T)\|^2_{L^2} - \|u_0\|^2_{L^2}\right] \; & =  
\; \E \left[ \|u(T)\|^2_{L^2}\right]  \;  -  \;  \E \left[\|u_0\|^2_{L^2}\right] \\
 & \; \leq \; \liminf_{ h \to 0}  \E \left[ \|u^h(T)\|^2_{L^2}\right]  
\;  -  \;  \liminf_{ h \to 0}   \|u^h_0\|^2_{L^2}
\end{aligned}
\end{equation*}
which finally leads to
\begin{equation}\label{eq:lsc}
 \E \left[ \|u(T)\|^2_{L^2} - \|u_0\|^2_{L^2}\right] \; 
\leq \; \liminf_{ h \to 0}  \E \left[ \|u^h(T)\|^2_{L^2} - \|u^h_0\|^2_{L^2}\right]
\end{equation}
finishing the proof.
\end{proof}

\begin{theorem}
\label{theo:convergence}
Let the initial data $u_0\in C^4(\mathcal{D})$ be deterministic and 
let Assumption \ref{bed:J} and the conditions \eqref{eq:sigmagrowth} 
and \eqref{eq:sigmalip} be satisfied. Then the solution $u^h$  of 
\eqref{eq:discstochnagumo} converges in $L^2(\Omega; L^2([0,T]; 
H_0^1(\mathcal{D}))) $ to the solution $u$ of \eqref{eq:stochnagumo} 
as $h \to 0$. 
\end{theorem}

\begin{proof}
Recall that we can write the discrete problem in integral 
form~\eqref{eq:discstochnagumointegrated} in a suggestive way, using 
an orthonormal basis $\lbrace e_i\rbrace_{i=1 \ldots N}$ of $P_N H_0^1$, as 
\begin{equation}
\label{eq:tested}
\begin{aligned}
\left(u^h(T), e_i \right)  \, & = \; \left(u^h_0,e_i\right) \; 
+ \; \nu \int_0^T \left\langle \Delta_R^h u^h, e_i \right\rangle~\txtd t 
+ \; \int_0^T \left( f(u^h) ,e_i \right) ~\txtd t\; \\
& + \;  \int_0^T g (u^h) \left(e_i, \txtd W^h(t)\right) \qquad i = 1, \ldots, N.
\end{aligned}
\end{equation}
The above priori estimates in Proposition \ref{prop:aprioriu}, imply the 
boundedness of the sequence $\Delta_R^h u^h $ in $L^2(\Omega \times [0,T]; 
H^{-1}(\mathcal{D}))$ and the boundedness of the sequence $g (u^h)$ in 
$L^2(\Omega \times [0,T] ; \mathcal{H})$. 
Hence there exists a subsequence, which we do not relabel, such that
\begin{equation}
\label{eq:weakc}
\begin{aligned}
u^h \rightharpoonup & \; u \quad \textup{in } L^2(\Omega; L^2([0,T]; 
H_0^1(\mathcal{D}))) \cap L^2\left(\Omega; L^{\infty}([0,T],L^2(\mathcal{D})\right) \\
\Delta_R^h u^h  \rightharpoonup & 
\; \zeta_1 \quad \textup{in } L^2(\Omega \times [0,T] ; H^{-1}(\mathcal{D})) \\
f(u^h)  \rightharpoonup & 
\; \zeta_2 \quad \textup{in } L^2(\Omega \times [0,T] ; L^2(\mathcal{D})) \\
g(u^h) \rightharpoonup  &  \; \widetilde{g} 
\quad \textup{in }  L^2(\Omega \times [0,T] ; \mathcal{H})
\end{aligned}
\end{equation}
We pass to the weak limit in \eqref{eq:tested} and get that for all $t \geq 0$ 
\begin{equation}
\label{eq:weaklimit}
\begin{aligned}
\left(u(T),  e_i \right)  \, & = \; \left(u_0, e_i \right) \; 
+ \; \nu \int_0^T \left\langle \zeta_1(t),  e_i \right\rangle \txtd t \\ 
  &+  \; \int_0^T \left( \zeta_2(t) ,  e_i \right) ~\txtd t\;  
	+ \;  \int_0^T \widetilde{g}(t)  \left( e_i, d W(t) \right) \; \qquad i = 1, \ldots, N. 
\end{aligned}
\end{equation}
It remains to identify the weak limit objects in \eqref{eq:weaklimit} with the 
objects in the stochastic Nagumo equation. We set in the rest of the proof $\nu=1$ 
for convenience as it does not change the argument. We start with identifying 
$\widetilde{g} = g (u)$. For this, note first that, by using the monotonicity 
property \eqref{eq:monotonicity} with $c = 0$, we infer that for any 
$\varphi \in L^2(\Omega \times [0,T];H_0^1(\mathcal{D}))$ holds
\begin{equation}
\label{eq:passtothelimithere}
\E  \left[ \int_0^T \left\langle  \Delta_R^h u^h + f(u^h) -  \Delta_R^h \varphi 
- f(\varphi), u^h - \varphi \right\rangle \, \txtd t \right]  +  
\E \left[ \int_0^T   \| g (u^h) - g (\varphi) \|^2_{\mathcal{H}}  \, \txtd t \right]  
\; \leq 0 
\end{equation}
We can split the first term into four terms and use the positivitiy of 
$\E \left[ \int_0^T   \| g (u^h) - g (\varphi) \|^2_{\mathcal{H}}  \, \txtd t \right]$ 
to get
\begin{equation}
\label{eq:fourterms}
\begin{aligned}
\E  &\left[ \int_0^T \left\langle  \Delta_R^h u^h + f(u^h) -  
\Delta_R^h \varphi - f(\varphi), u^h - \varphi \right\rangle \, \txtd t \right]\\
&=  \E  \left[ \int_0^T \left\langle \Delta_R^h u^h + f(u^h), 
u^h\right\rangle ~\txtd t \right] \;
+ \; \E  \left[  \int_0^T \left\langle\Delta_R^h \varphi + f(\varphi), 
\varphi \right\rangle ~\txtd t \right] \\
\; & - \;\E  \left[  \int_0^T \left\langle\Delta_R^h u^h + f(u^h), \varphi 
\right\rangle ~\txtd t \right] \;
- \; \E  \left[  \int_0^T \left\langle\Delta_R^h \varphi + f(\varphi), u^h
 \right\rangle ~\txtd t \right] \;  \leq  \; 0 .
\end{aligned}
\end{equation}
By weak convergence, 
\begin{equation}
\label{eq:fourterms2}
\begin{aligned}
 &\int_0^T \left\langle\Delta_R^h \varphi + f(\varphi), \varphi 
\right\rangle ~\txtd t \;  \longrightarrow \; \int_0^T \left\langle\Delta 
\varphi + f(\varphi), \varphi \right\rangle ~\txtd t  \\
& \int_0^T \left\langle\Delta_R^h u^h + f(u^h), \varphi \right\rangle ~\txtd t 
\;  \longrightarrow \; \int_0^T \left\langle \zeta_1(t) + \zeta_2(t), 
\varphi \right\rangle ~\txtd t \; \\
& \int_0^T \left\langle\Delta_R^h \varphi + f(\varphi), u^h \right\rangle 
~\txtd t  \; \longrightarrow  \; \int_0^T \left\langle\Delta \varphi 
+ f(\varphi), u \right\rangle ~\txtd t \\
& \int_0^T   \Big( g (u^h),  g (\varphi) \Big)_{\mathcal{H}}  \, ~\txtd t 
\longrightarrow  \; \int_0^T   \Big( \tilde{g}(t),  g (\varphi) \Big)_{\mathcal{H}}  \, 
~\txtd t
\end{aligned}
\end{equation}
so the last three terms in \eqref{eq:fourterms} pass to the limit and preserve 
the sign in \eqref{eq:passtothelimithere}. For the first term of \eqref{eq:fourterms}, 
we employ that by semicontinuity of the norm and Lemma \ref{lemma:lscnorm}, we can 
relate solutions $u^h$ to \eqref{eq:discstochnagumo} with solutions $u$ to 
\eqref{eq:stochnagumo} as follows
\begin{equation}
\label{eq:uselschere}
\begin{aligned}
\E &\left[ \int_0^T \left\langle \Delta u (t), u(t)\right\rangle ~\txtd t \right] 
   + \; \E \left[ \int_0^T \left( f(u(t)), u(t)\right) ~\txtd t \right]  \;  
	+ \E \left[ \int_0^T g(u(t))^2 ~\txtd t \right] \\
\leq &\; \liminf_{h \to 0} \E \left[ \int_0^T \left\langle \Delta_R^h u^h(t), 
u^h(t)\right\rangle ~\txtd t \right]  \; + \;  \liminf_{h \to 0} \E \left[  
\int_0^T \left( f(u^h(t)), u^h(t)\right) ~\txtd t \right]  \\ 
+ & \; \liminf_{h \to 0} \E  \left[ \int_0^T g(u^h(t))^2 ~\txtd t \right] .
\end{aligned}
\end{equation}
Consequently also for the first term in \eqref{eq:fourterms} the sign is 
preserved in the limit. Passing to the limit in \eqref{eq:passtothelimithere}, we get
 \begin{equation}
\label{eq:identifynoise}
 \E  \left[ \int_0^T \left\langle  \zeta_1 + \zeta_2 -  \Delta \varphi 
- f(\varphi), u - \varphi \right\rangle ~\txtd t \right]  
+ \E \left[ \int_0^T   \| \widetilde{g} - g (u) \|^2_{\mathcal{H}}  ~\txtd t \right]  
\; \leq \; 0 .
 \end{equation}
Choosing $u = \varphi$ in \eqref{eq:identifynoise}, we deduce $\widetilde{g} = g (u)$.

It remains to identify the limit objects $\zeta_1$ and $\zeta_2$ to prove that 
$ u := \lim_{h \to 0} u^h$ is indeed a solution to the stochastic Nagumo equation.
To this aim, notice first that \eqref{eq:identifynoise} implies 
\begin{equation}
\label{eq:identifyoperators}
 \E  \left[ \int_0^T \left\langle  \zeta_1(t) + \zeta_2(t) -  
\Delta \varphi(t) - f(\varphi(t)), u(t) - \varphi(t) \right\rangle ~\txtd t 
\right]   \; \leq  \; 0 .
\end{equation}
Now we take $\theta > 0$ and define another test function $w$ via 
\begin{equation}
  \theta w(t) = u(t) - \varphi(t)
\end{equation}
with $ \varphi(t)$ the test function used in~\eqref{eq:identifyoperators}. 
As $w$ is an admissible test function in $L^2\left( \Omega \times [0,T]; 
H_0^1(\mathcal{D})\right)$, we can employ it in \eqref{eq:identifyoperators} 
instead of $\varphi$. We obtain
\begin{equation}
\label{eq:limitcalc1}
\begin{aligned}
  \E \left[\int_0^T \left\langle \zeta_1(t) + \zeta_2(t) -\Delta \left(u(t) 
	- \theta w(t)\right) -  f\left(u(t) - \theta w(t)\right), w(t)\right\rangle 
	~\txtd t \right] \leq 0 \; 
\end{aligned}
\end{equation}
As $\theta \mapsto  \langle \Delta (u - \theta w), w\rangle $ and $\theta \mapsto  
\langle f (u - \theta w), w\rangle $ are continuous from $\mathbb{R} \to \mathbb{R}$, it is 
admissible to pass to the limit $\theta \to 0$ and we reach
\begin{equation}
\label{eq:limitcalc2}
\E \left[\int_0^T \left\langle \zeta_1(t) + \zeta_2(t) - \Delta u(t) -  f(u(t)) , 
w(t)\right\rangle ~\txtd t \right] \leq 0 \qquad \textup{ for any } w \in L^2\left( 
\Omega \times [0,T]; H_0^1(\mathcal{D})\right).
\end{equation}
Since $w$ is arbitrary, the left hand side must vanish, hence $ \zeta _1 + \zeta_2  
= \Delta u + f(u) $. Setting now $w=u$ we identify $\zeta_2 = f(u)$. Plugging this 
result into \eqref{eq:identifyoperators} gives $\zeta_1 = \Delta u$.
\end{proof}

\subsection{The cut-off error}
\label{ssec:cutoff}

The above Theorem \ref{theo:convergence}  dealt with the error between the solutions to 
the lattice and the continuum model in a bounded interval  $\mathcal{D}=[-L,L]$. In this 
section we investigate the error that we make when truncating a solution living on the whole 
real line. The idea is to select $L$ large enough
so that a travelling front is contained in $\cD$ for the time of interest. Then the error outside
of $\cD$ is small as any classical Nagumo front decays exponentially near the 
two endstates.   

We start with some notation. To compare a solution to the Nagumo PDE on the real line~\eqref{eq:nagumo} 
with a solution $v$ of the finite-domain Nagumo PDE  \eqref{eq:nagumoPDED}, we set, as always in this 
work, the domain to the symmetric interval $[-L,L]$ and extend $v$ on  $\mathbb{R} \setminus \cD$ by 
extending suitable Dirichlet boundary conditions, i.e., we denote by $v$ now the solution to
\begin{equation} \label{eq:nagumohintencut}
\begin{aligned}
\partial_t v  &= \nu \partial_{x}^2 v + f(v) \qquad (t,x) \in \mathbb{R}_+ \times \mathcal{D}\\
 v(x, t) &= 0 \qquad \textup{for all} x \in (-\infty , -L] \\
 v(x,t) &= 1\qquad \textup{for all} x \in (-\infty , -L] \\.
\end{aligned}
\end{equation} 
In the same way, we extend a solution $u$ to \eqref{eq:nagumoSPDED} to $\mathbb{R} \setminus \cD$ 
via Dirichlet boundary conditions 
\begin{equation}
\label{eq:stochnagumohintencut}\tag{{SNag$\mathcal{D}$}}
\begin{aligned}
  \txtd u(t) &= [ \nu \partial_{x}^2 u(t) + f(u(t))] \txtd t + g(u(t))~\txtd 
	W(t) \qquad \qquad \textup{ on } \mathcal{D} \times [0,T]\\
	 u(-x, t) &= 0 \qquad \textup{for all} x \in (-\infty , -L] \\\\
 u(x,t) &=1 \qquad \textup{for all} x \in (-\infty , -L] \\.
\end{aligned}
\end{equation}
Analogously, we extend the solution $u^h$ to \eqref{eq:discstochnagumo}, which  is defined per 
definition only on $\mathcal{D} = [-L,L]$, by a constant  $u^h \equiv 1$ on $[L, \infty)$ and 
$u^h \equiv 0$ on $[-\infty, -L)$.  

\begin{proposition}
\label{prop:errorcontrol}
Assume $u_0 = V^{{\textnormal{TW}}}(\cdot-ct)=V_0$. Let $u^h$ 
solve \eqref{eq:discstochnagumo} and $V$ solve \eqref{eq:nagumo}.   
Then, given any fixed $T>0$ and $\epsilon>0$, there exists sufficiently 
small non-vanishing noise (in the 
sense that  $\|Q\|_{\mathcal{H}}^2 \leq \tilde{\epsilon}$), $h>0$, and $L>0$ 
such that for all $t \leq T$
\begin{equation}\label{eq:cutofferror}
  \E \left[ \| u^h(t) - V(t) \|_{L^2(\mathbb{R})}^2\right] \leq \epsilon .
\end{equation}
\end{proposition}

\begin{proof}
We want to control the cut-off error for terms of the form 
\benn
\E \left[ \| u^h(t) - V(t) \|_{L^2(\mathbb{R})}^2\right].
\eenn
The strategy is to split the cut-off error into several parts, making use 
of the deterministic solutions $v$ and $V$. 

Applying the triangle inequality yields
\begin{equation}
\label{eq:startsplit}
\begin{aligned}
\E \left[ \| u^h(t) - u(t) \|_{L^2(\mathbb{R})}^2\right] 
\leq&  \; \E \left[ \| u^h(t) - u(t) \|_{L^2(\mathcal{D})}^2\right] + 
\E \left[ \| u(t) - v(t) \|_{L^2(\mathcal{D})}^2\right]  \\
&+ \; \| v(t) -  V(t) \|_{L^2(\mathbb{R})},
\end{aligned}
\end{equation}
where employed that, by definition of the constant continuations in \eqref{eq:nagumohintencut} 
and \eqref{eq:stochnagumohintencut} satisfy
\benn
\E \left[ \| u^h(t) - u(t) \|_{L^2(\mathbb{R} \setminus \mathcal{D})}^2\right] 
= \E \left[ \| u(t) - v(t) \|_{L^2(\mathbb{R} \setminus \mathcal{D})}^2\right] = 0.
\eenn
The control of $ \| v(t) -  V(t) \|_{L^2(\mathbb{R})}$ is established in 
Lemma \ref{lemma:errorcontroldet} below. Lemma \ref{lemma:errorcontroldetstoch} estimates 
$\E \left[ \| u(t) - v(t) \|_{L^2(\mathcal{D})}^2\right] \leq \epsilon $. This 
yields the desired result, equation \eqref{eq:cutofferror}.
\end{proof}

To establish the relevant auxiliary results, we start with the 
truncation error for the deterministic equation.

\begin{lemma}
\label{lemma:errorcontroldet} 
Let $V$ solve the Nagumo PDE on the real line~\eqref{eq:nagumo} and let 
$V^{{\textnormal{TW}}}$ be the deterministic travelling front solution to \eqref{eq:nagumo}.   
Let $v$ be the solution to  \eqref{eq:nagumohintencut}.
Let the initial data $v_0$ be the truncation of $V^{{\textnormal{TW}}}$. Then, for all 
$t < T$, where $T$ depends on the size of the domain and the speed of the wave, we have
\begin{equation}
 \| V(t) - v(t) \|_{L^2(\mathbb{R})} \leq \epsilon.
\end{equation}
\end{lemma}

\begin{proof}
We focus on such $v$ which are close to travelling front solutions, i.e., which 
satisfy suitable boundary conditions. So we always formally extend $v$ to $\mathbb{R}$ by setting 
$v(x,t) = 1$ for all $x \in [L, \infty)$ and $v(x,t) = 0$ for all $x\in (-\infty , -L]$. 
See equation \eqref{eq:lengthofl} for the choice of $L$.

We start both equations with the same initial data $v_0$, which forms a traveling wave 
with a front near zero. Due to the boundary conditions at infinity, we know that 
$V(-\infty,t) = 0$ and $V(\infty,t) = 1$. By classical theory, see 
e.g.~\cite{AronsonWeinberger}, $0$ and $1$ are hyperbolic stable steady endstates 
of \eqref{eq:nagumo}. Therefore, the decay of the traveling waves is exponential as 
it approaches the steady states, in other words, there  exists $\e$ such that 
$\txte^{\e x} \cdot V(x,t) \to 0$  for $x \to -\infty$, and analogously for 
$x \to + \infty$ and decay towards $V\equiv 1$. Therefore, as long as the transition 
part of the traveling wave is far away from the boundary of $\mathcal{D}$, 
\begin{equation}
\label{eq:lengthofl}
\exists L_0 \in \mathbb{R} \; : \forall L \geq L_0 \; : \quad
 \| V(x,t) - 1 \|_{L^{2}((L, \infty))} \leq \epsilon \quad , \qquad  
\| V(x,t) - 0 \|_{L^{2}((-\infty,-L))} \leq \epsilon .
\end{equation}
Consequently, by the boundary conditions of the  Nagumo 
PDE~\eqref{eq:nagumohintencut} on the interval, we conclude that up to a time $T_0$, when the transition 
part of the traveling waves has not yet reached the boundary of the interval $\mathcal{D}$,
\begin{equation}\label{eq:erroratinfinity}
 \| V(x,t) -  v (x,t) \|_{L^2( \mathbb{R} \setminus [-L - \delta , L + \delta])} \leq 2 \epsilon .
\end{equation}
Second, we investigate the error at the boundary of $\mathcal{D}$. First, we look at the 
positive boundary point of $\mathcal{D}$, i.e. the point $L \in \mathbb{R}$. In a 
neighborhood $B_{\delta}(L)$ we have $f(1-\delta) = (1-\delta-a)(1- \delta) \delta 
\leq  \delta  $ and $V(x,t)$ is almost constant so by spatial regularity, we have 
$\max_x \partial_{x}^2 V(x,t) \leq \tilde{\epsilon}$, and as $V \equiv 1$ in 
$[L , L+\delta]$. This gives
\begin{equation}
\label{eq:deltaumgebungcalc}
\begin{aligned}
 \| V(x,t) - v(x,t) \|_{L^2(B_{\delta}(L))}^2 =&  \int_L^{L+\delta} \left( 
\nu \partial_{x}^2 V(x,t) + f(V(x,t))  - 0 \right)^2 ~\txtd x \\
 &+ \int_{L-\delta}^L \left( \nu \partial_{x}^2 V(x,t) + f(V(x,t)) -  
\nu \partial_{x}^2 v(x,t) - f(v(x,t)) \right)^2~ \txtd x \\
\leq& \; \int_L^{L+\delta} \left( \nu \max_{x \in [L, L + \delta] } 
\partial_{x}^2 V(x,t) + \max_{x \in [L - \delta , L] } f(V(x,t))  
- 0 \right)^2 ~\txtd x \\
&+ \;\delta \cdot \left( \nu  \max_{x \in [L - \delta , L] } 
\left( \partial_{x}^2 V(x,t)  - \partial_{x}^2 v(x,t) \right)  +  
\max_{x \in [L - \delta , L]} (f(V) - f(v)) \right)^2\\
\leq& \delta \nu^2 \tilde{\epsilon}^2 + \delta^3 + 
\; 3 \delta \nu^2 \left( \max_{x \in [L - \delta , L] }( \partial_{x}^2 
V(x,t)  - \partial_{x}^2 v(x,t)) \right)^2  \; + \; 3 \delta^2   \\
 \leq& \; \epsilon
 \end{aligned}
\end{equation}
where we employed the regularity in space of solutions to \eqref{eq:nagumo} 
and \eqref{eq:nagumohintencut} again to infer that 
$ \max_{x \in [L - \delta , L] }( \partial_{x}^2 V(x,t)  - \partial_{x}^2 v(x,t)) 
\leq c$ is a finite quantity. Analogous reasoning holds for the $\delta$-neighbourhood 
around $-L$. 
Therefore
\begin{equation}
\label{eq:deltaumgebungv}
 \| V(x,t) - v(x,t) \|_{L^2(B_{\delta}(L))} \leq \epsilon \quad ,\quad  
\| V(x,t) - v(x,t) \|_{L^2(B_{\delta}(-L))} \leq \epsilon .
\end{equation}
Third, we look at the difference between the two solutions in the interior 
of the interval $\mathcal{D}$,  $\| V(x,t) -  v (x,t) \|_{L^2([-L + \delta , 
L - \delta])}$. By construction, this difference is zero at time $t=0$, 
as we started with the same wave as initial condition. 
Moreover, as
\begin{equation}
\label{eq:interiorcalc}
\begin{aligned}
 \| \partial_t V(x,t) - \partial_t v(x,t) \|_{L^2([-L + \delta , L - \delta])}^2 
 =& \int_{-L+\delta}^{L-\delta} \left( \nu \partial_{x}^2 V(x,t) + f(V(x,t)) -  
\nu \partial_{x}^2 v(x,t) - f(v(x,t)) \right)^2~ \txtd x \\
 \leq& \; 3 \nu \int_{-L+\delta}^{L-\delta} \left(\partial_{x}^2 (V(x,t) - 
v(x,t))\right)^2~ \txtd x \\
 &+ \; 3\int_{-L+\delta}^{L-\delta} \left( f(V(x,t)) -  f(v(x,t)) \right)^2~\txtd x \\
 \end{aligned}
\end{equation}
so $ \| \partial_t (V(x,t) -  v(x,t)) \|_{L^2([-L + \delta , L - \delta])} =0 $ for 
times $t \leq T_0$, i.e., for all $t$ small enough. As we have already controlled the 
error at the boundary by $\epsilon$, by continuity of the solutions $V$ and $v$, 
we can conclude $ \| \partial_t (V(x,t) -  v(x,t)) \|_{L^2([-L + \delta , L - \delta])} 
\leq \epsilon $. Using~\eqref{eq:erroratinfinity}, \eqref{eq:deltaumgebungv} and 
\eqref{eq:interiorcalc} finishes the proof.
\end{proof}

\subsection{Small noise estimate}\label{ssec:smallnoise}
To complete the proof of Proposition \ref{prop:errorcontrol}, it remains to estimate 
the remaining term in \eqref{eq:startsplit}, namely 
$\E \left[ \| u(t) - v(t) \|_{L^2(\mathcal{D})}\right]$. This term
quantifies the error between 
the deterministic solution on a bounded interval for~\eqref{eq:nagumohintencut} and the stochastic solution
on a bounded interval for~\eqref{eq:stochnagumohintencut}. 

\begin{lemma}
\label{lemma:errorcontroldetstoch}
Let $v$ be a solution to \eqref{eq:nagumohintencut} and $u$ a solution to 
\eqref{eq:stochnagumohintencut}. Then, given any $\epsilon>0$ and any fixed finite $T>0$, 
there exists a sufficiently small (non-vanishing) noise with $\|Q\|_{\mathcal{H}}^2 
\leq \tilde{\epsilon}$, such that
\begin{equation}
\sup_{t\in{[0,T]}}\E \left[ \| v(t) - u(t) \|_{L^2(\mathcal{D})}^2\right] \leq \epsilon
\end{equation}
\end{lemma}

\begin{proof}
We assume that both equations start with the same deterministic initial data $u_0$, therefore 
at time $t=0$, both equations satisfy the same boundary conditions and the error is zero.
At time $t$,
%we still have  $v(-L, t) = 0 = $ and $v(L,t) = 1$, and 
we write the error between the deterministic 
and the stochastic solution via the mild solution expression as
\begin{equation}\label{eq:milddifference}
\begin{aligned}
  u(t) - v(t)   
  =& \; S(t) (u_0 - u_0)  + \;  \int_0^t S(t-s)  \left(f(u(s)) - f(v(s)) \right)~\txtd s  \\
  &+  \; \int_0^t S(t-s) \tilde{g}(u(s)) ~\txtd\tilde{W}(s).  \\
 \end{aligned}
\end{equation}
with $\tilde{W}$ a cylindrical Wiener Process on $L^2(\mathcal{D})$ and 
\begin{equation}
  \tilde{g}(u) \phi = g(u(s)) \sqrt{Q} \phi
\end{equation}
where  $\sqrt{Q}$ is the positive definite square root of the the covariance 
operator $Q$ of the Wiener Process. We use the boundedness of the heat semigroup, 
the local Lipschitz continuity of $f$ from Lemma \ref{lemma:liploc}, It\^{o}'s Isometry 
and Gronwall's equality to get 
\begin{equation}
\begin{aligned}
 \E \left[ \| u(t) - v(t) \|_{L^2(\mathcal{D})}^2\right] 
   =&  \; \E \left[ \int_{\mathcal{D}} \left( \int_0^t S(t-s)  \left(f(u(s)) - f(v(s)) 
	\right) ~\txtd s \right)^2 ~\txtd x \right]  \\
  &+  \E \left[ \int_{\mathcal{D}}\left( \int_0^t S(t-s) g(u(s))
	~\txtd W(s) \right)^2 ~\txtd x \right] \\
   \leq&  \; c_a^2 K_M \;  \| S(t-s)\|_{\infty}  \E \left[ 
	\int_{\mathcal{D}}\left( \int_0^t  | u(s) - v(s)| ~\txtd s  \right)^2 
	~\txtd x \right]  \\
  \;  &+ \;  \sigma \,  \cdot  \; \E \left[ \int_{\mathcal{D}}\int_0^t S(t-s)^2  g(u(s))^2 
	 ~\txtd s ~\txtd x \right] \\
  \leq& \; \sigma \,  \cdot  \; \exp \left( \int_0^t c_a K_M \| S(t-s)\|_{\infty}
	~\txtd s \right)  \E \left[ \int_{\mathcal{D}}\int_0^t S(t-s)^2  g(u(s))^2  
	~\txtd s ~\txtd x \right]\\
  =& \sigma \cdot c \;  \E \left[ \int_{\mathcal{D}}\int_0^t S(t-s)^2  g(u(s))^2 
	~\txtd s ~\txtd x \right]
 \end{aligned}
\end{equation}
where we used the notation $\sigma := \| Q \|_{\mathcal{H}}^2$
and $c = c \left(t,c_a, \|u\|_{H_0^1(\mathcal{D})}, \|v\|_{H_0^1(\mathcal{D})},  
\| S(t-s)\|_{\infty} \right)$.
Under the assumption $g(u(s)) \in L^2(\Omega, L^2([0,T] \times D)) $, we can conclude 
that for finite times $t \leq T$, 
\begin{equation}
\label{eq:finitetimesmallnoise}
\begin{aligned}
 \E \left[ \| u(t) - v(t) \|_{L^2(\mathcal{D})}^2\right] 
      \leq& \; \| Q \|_{\mathcal{H}}^2 \cdot c\left(T, b, 
			c_a, \|u\|_{H_0^1(\mathcal{D})}, \|v\|_{H_0^1(\mathcal{D})},  
			\| S(t-s)\|_{\infty} \right)) 
 \end{aligned}
\end{equation}
The choice of $ \| Q \|_{\mathcal{H}}^2 \leq \frac{\epsilon}{c} $ with the 
$c$ from \eqref{eq:finitetimesmallnoise} finally gives 
$ \E \left[\| u(T) - v(T) \|_{L^2(\mathcal{D})}^2 \right] \leq \epsilon$, finishing the proof.
\end{proof}

\subsection{Proof of the main theorem}
\label{sec:maintheorem}

%%%%%%%%%%%%%%% MAIN THEOREM %%%%%%%%%%%%%%%%%%%%%%%%%%%

We can now finally prove that the  solution of the stochastic 
LDE \eqref{eq:discstochnagumo} is likely to be close to the 
traveling wave. The below theorem, is the same as Theorem \ref{thm:intro} from the introduction, 
with the additional specifications on what sufficiently small noise means. 

\begin{theorem}
\label{theo:main}
Let $V^{\textnormal{TW}}=V^{\textnormal{TW}}(t,x)$ be a travelling front solution 
to~\eqref{eq:nagumo}, $u^h_0$ be deterministic lattice initial data and $u^h$ a solution 
to ~\eqref{eq:discstochnagumo} on the interval $\mathcal{D}:=[-L,L]$ to the initial data 
$u^h_0 \equiv u^h(0)$. Suppose $L >0$ large enough, $\delta>0$, $T>0$, and $\tilde{\varepsilon}>0$ 
are given, and the initial data $u^h(0)$  satisfies
\begin{equation}
\| u^h(0) - V^{{\textnormal{TW}}}(0,\cdot)\|^2_{L^2(\mathbb{R})} < \varepsilon .
\end{equation}
Then there exists $\varepsilon>0$ and $c\in\mathbb{R}$ such that,
for  sufficiently small $h>0$, 
and $\|Q\|_{\mathcal{H}}^2 \leq \tilde{\epsilon}$, 
the solution $u^h$ to~\eqref{eq:discstochnagumo} satisfies
\begin{equation}
 \Proba \left[ \sup_{t \in [0,T]} \| u^h(t) - V^{{\textnormal{TW}}}(\cdot-ct)  
\|_{L^2(\mathbb{R})} > \delta \right] \; \leq \;  \tilde{\varepsilon}.
\end{equation}
\end{theorem}

\begin{proof}
As we have outlined in the beginning of this section, there exists 
a solution to equation~\eqref{eq:discstochnagumo}, which is an adapted 
process with a continuous version. The question is whether this (discrete-in-space) 
solution is likely to be close to the deterministic travelling wave. 
We will show that this is indeed true by comparing $u^h$ to solutions to 
several intermediate problems. 

To set up notation, let us now denote by $u^h(t)$ be the piecewise linear 
extension of the solution to the stochastic LDE \eqref{eq:discstochnagumo} 
to the whole interval $\mathcal{D}$. Moreover we extend the solution as 
$u^h(t) \equiv 1$ and  $u^h(t) \equiv 0$ from the two boundary points 
of $\mathcal{D}$ to $\pm \infty$. Consider the adapted stochastic process 
$e_t := \| u^h(t) - V^{{\textnormal{TW}}}(\cdot-ct)  \|_{L^2(\mathcal{D})} $. 
By the properties of $u^h$, $e_t$ defines a martingale whose trajectories 
are continuous almost surely. We can then estimate by Doob's inequality
\begin{equation}
\label{eq:stabilitystart}
\begin{aligned}
 \Proba \left[ \sup_{t \in [0,T]} \| u^h(t) - V^{{\textnormal{TW}}}(\cdot-ct)
\|_{L^2(\mathbb{R})} > \delta \right]  
 &\leq \frac{1}{\delta^2} \E \left[ \| u^h(T) - V^{TW}(\cdot-cT) \|_{L^2(\mathbb{R})}^2. 
\right]
\end{aligned}
\end{equation}
We split the error between the stochastic LDE and the travelling wave front 
into three parts, using the linearity of the expectation, i.e., we use
\benn
\|u^h-V^{\textnormal{TW}}\|_{L^2}=\|u^h-V+V-V^{\textnormal{TW}}\|_{L^2}
\leq \|u^h-V\|_{L^2}+\|V-V^{\textnormal{TW}}\|_{L^2},
\eenn
and then take expectations. By Proposition~\ref{prop:errorcontrol}, the first 
term goes to zero as $h \to 0$, while the second term is small by the standard
deterministic local asymptotic $L^2$-stability of the travelling wave front in 
the deterministic setting, i.e., the front is known to be deterministically 
stable for the Nagumo equation~\cite{Chen1,KuehnBook1,Sandstede1}.
\end{proof}

\subsection*{Acknowledgment}
 CK~was supported by a Lichtenberg Professorship of the VolkswagenStiftung as well as by the Deutsche 
Forschungsgemeinschaft (DFG) via the CRC/TR109 ``Discretization in Geometry and Dynamics''.

%%%%%%%%%%%%%%%%%%%%%%%%%%%%%%%%%%%%%%%%%%%%%%%%%%%%%%%%%%%%%%%%%%%%%%%%%%%%%%%%%%%%%%%%%%%%%%%%%%%%%


\begin{thebibliography}{10}

\bibitem{AchleitnerKuehn2}
F.~Achleitner and C.~Kuehn.
\newblock Analysis and numerics of travelling waves for asymmetric fractional
  reaction-diffusion equations.
\newblock {\em Comm. Appl. Ind. Math.}, 6(2):1--25, 2015.

\bibitem{AllenCahn}
S.M. Allen and J.W. Cahn.
\newblock A microscopic theory for antiphase boundary motion and its
  application to antiphase domain coarsening.
\newblock {\em Acta Metallurgica}, 27(6):1085--1905, 1979.

\bibitem{AronsonWeinberger}
D.G. Aronson and H.F. Weinberger.
\newblock Nonlinear diffusion in population genetics, combustion, and nerve
  pulse propagation.
\newblock In {\em Partial Differential Equations and Related Topics}, volume
  446 of {\em Lecture Notes in Mathematics}, pages 5--49. Springer, 1974.
	
\bibitem{BatesChenChmaj}
P.W. Bates, X.~Chen, and A.J. Chmaj.
\newblock Traveling waves of bistable dynamics on a lattice.
\newblock {\em SIAM J. Math. Anal.}, 35(2):520--546, 2003.

\bibitem{BatesFifeRenWang}
P.W. Bates, P.C. Fife, X.~Ren, and X.~Wang.
\newblock Traveling waves in a convolution model for phase transitions.
\newblock {\em Arch. Rat. Mech. Anal.}, 138(2):105--136, 1997.

\bibitem{BerglundGentz10}
N.~Berglund and B.~Gentz.
\newblock Sharp estimates for metastable lifetimes in parabolic {SPDEs}:
  Kramers' law and beyond.
\newblock {\em Electronic J. Probability}, 18(24):1--58, 2013.

\bibitem{BovierGeldhauser}
A.~Bovier and C.~Geldhauser.
\newblock The scaling limit of a particle system with long-range interaction.
\newblock {\em Markov Process. Related Fields}, 23(4):515--552, 2017.

\bibitem{CantrellCosner}
R.S. Cantrell and C.~Cosner.
\newblock {\em Spatial Ecology via Reaction-Diffusion Equations}.
\newblock Wiley, 2004.

\bibitem{cerrai1999}
S.~Cerrai.
\newblock Smoothing properties of transition semigroups relative to {SDE}s with
  values in {B}anach spaces.
\newblock {\em Probab. Theory Related Fields}, 113(1):85--114, 1999.

\bibitem{Chen1}
X.~Chen.
\newblock Existence, uniqueness, and asymptotic stability of travelling waves
  in nonlocal evolution equations.
\newblock {\em Adv. Differential Equations}, 2:125--160, 1997.

\bibitem{Chow}
P.-L. Chow.
\newblock {\em Stochastic Partial Differential Equations}.
\newblock {Chapman \& Hall / CRC}, 2007.

\bibitem{CrossHohenberg}
M.C. Cross and P.C. Hohenberg.
\newblock Pattern formation outside of equilibrium.
\newblock {\em Rev. Mod. Phys.}, 65(3):851--1112, 1993.

\bibitem{EichingerGnannKuehn}
K.~Eichinger, M.~Gnann, and C.~Kuehn.
\newblock Multiscale analysis for traveling-pulse solutions to the stochastic
  {FitzHugh-Nagumo} equations.
\newblock {\em preprint}, pages 1--36, 2020.

\bibitem{ElmerVanVleck1}
C.~Elmer and E.S.~Van Vleck.
\newblock Traveling wave solutions for bistable differential-difference
  equations with periodic diffusion.
\newblock {\em SIAM J. Appl. Math.}, 61(5):1648--1679, 2001.

\bibitem{ElmerVanVleck2}
C.~Elmer and E.S.~Van Vleck.
\newblock Dynamics of monotone travelling fronts for discretizations of
  {N}agumo {PDE}s.
\newblock {\em Nonlinearity}, 18(4):1605--1628, 2005.

\bibitem{ErmentroutTerman}
G.B. Ermentrout and D.H. Terman.
\newblock {\em Mathematical Foundations of Neuroscience}.
\newblock Springer, 2010.

\bibitem{Evans}
L.C. Evans.
\newblock {\em Partial Differential Equations}.
\newblock AMS, 2002.

\bibitem{FitzHugh}
R.~FitzHugh.
\newblock Mathematical models of threshold phenomena in the nerve membrane.
\newblock {\em Bull. Math. Biophysics}, 17:257--269, 1955.

\bibitem{GarciaOjalvoSancho}
J.~Garcia-Ojalvo and J.~Sancho.
\newblock {\em Noise in Spatially Extended Systems}.
\newblock Springer, 1999.

\bibitem{GriffithsStuartYee}
D.F. Griffiths, A.M. Stuart, and H.C. Yee.
\newblock Numerical wave propagation in an advection equation with a nonlinear
  source term.
\newblock {\em SIAM J. Numer. Anal.}, 29(5):1244--1260, 1992.

\bibitem{Gyongy}
I.~Gy{\"{o}}ngy.
\newblock Lattice approximations for stochastic quasi-linear parabolic partial
  differential equations driven by space-time white noise {I}.
\newblock {\em Potential Anal.}, 9:1--25, 1998.

\bibitem{HamsterHupkes}
C.H.S. Hamster and H.J. Hupkes.
\newblock Stability of travelling waves for reaction-diffusion equations with
  multiplicative noise.
\newblock {\em arXiv:1712.00266}, pages 1--, 2017.

\bibitem{HodgkinHuxley4}
A.L. Hodgkin and A.F. Huxley.
\newblock A quantitative description of membrane current and its application to
  conduction and excitation in nerve.
\newblock {\em J. Physiol.}, 117:500--544, 1952.

\bibitem{HupkesSandstede2}
H.J. Hupkes, D.~Pelinovsky, and B.~Sandstede.
\newblock Propagation failure in the discrete {Nagumo} equation.
\newblock {\em Proc. Amer. Math. Soc.}, 139:3537--3551, 2011.

\bibitem{HupkesSchouten}
H.J. Hupkes and W.M. Schouten.
\newblock Nonlinear stability of pulse solutions for the discrete
  {F}itz{H}ugh-{N}agumo equation with infinite range interactions.
\newblock {\em preprint}, pages 1--, 2018.

\bibitem{HupkesvanVleck}
H.J. Hupkes and E.S. van Vleck.
\newblock Travelling waves for complete discretizations of reaction diffusion
  systems.
\newblock {\em J. Dyn. Diff. Equat.}, 28(3):955--1006, 2016.

\bibitem{InglisMacLaurin}
J.~Inglis and J.~MacLaurin.
\newblock A general framework for stochastic traveling waves and patterns, with
  application to neural field equations.
\newblock {\em SIAM J. Appl. Dyn. Syst.}, 15(1):195--234, 2016.

\bibitem{Jentzen}
A.~Jentzen.
\newblock Pathwise numerical approximation of {SPDEs} with additive noise under
  non-global {Lipschitz} coefficients.
\newblock {\em Potential Anal.}, 31(4):375--404, 2009.

\bibitem{KaratzasShreve}
I.~Karatzas and S.E. Shreve.
\newblock {\em Brownian Motion and Stochastic Calculus}.
\newblock Springer, 1991.

\bibitem{Keener}
J.P. Keener.
\newblock Propagation and its failure in coupled systems of discrete excitable
  cells.
\newblock {\em SIAM J. Appl. Math.}, 47:556--572, 1987.

\bibitem{KuehnBook1}
C.~Kuehn.
\newblock {\em PDE Dynamics: An Introduction}.
\newblock SIAM, 2019.

\bibitem{KuehnSPDEwaves}
C.~Kuehn.
\newblock Travelling waves in monostable and bistable stochastic partial
  differential equations.
\newblock {\em Jahresbericht der Deutschen Mathematiker-Vereinigung}, pages
  1--35, 2019.
\newblock accepted / to appear.

\bibitem{KuehnKuerschner}
C.~Kuehn and P.~K{\"u}rschner.
\newblock Combined error estimates for local fluctuations of {SPDEs}.
\newblock {\em arXiv:1611.04629}, pages 1--23, 2016.

\bibitem{KuehnNeamtu2}
C.~Kuehn and A.~Neamtu.
\newblock Dynamics of stochastic reaction-diffusion equations.
\newblock In H.~Lisei and W.~Grecksch, editors, {\em Finite and Infinite
  Dimensional Stochastic Equations with Applications to Physics}, pages 1--55.
  Wiley, 2020.

\bibitem{KuehnNeamtuPein}
C.~Kuehn, A.~Neamtu, and A.~Pein.
\newblock Random attractors for stochastic partly dissipative systems.
\newblock {\em arXiv:1906.08594}, pages 1--25, 2019.

\bibitem{Lang}
E.~Lang.
\newblock A multiscale analysis of traveling waves in stochastic neural fields.
\newblock {\em SIAM J. Appl. Dyn. Syst.}, 15(3):1581--1614, 2016.

\bibitem{LiuRoeckner1}
W.~Liu and M.~Roeckner.
\newblock {\em Stochastic Partial Differential Equations: An Introduction}.
\newblock Springer, 2015.

\bibitem{LordRougement}
G.J. Lord and J.~Rougement.
\newblock A numerical scheme for stochastic {PDEs} with {Gevrey} regularity.
\newblock {\em IMA J. Numer. Anal.}, 24:587--604, 2004.

\bibitem{LordThuemmler}
G.J. Lord and V.~Th{\"u}mmler.
\newblock Computing stochastic traveling waves.
\newblock {\em SIAM J. Sci. Comput.}, 34(1):B24--B43, 2012.

\bibitem{Mallet-Paret}
J.~Mallet-Paret.
\newblock The global structure of traveling waves in spatially discrete
  dynamical systems.
\newblock {\em J. Dyn. Diff. Eq.}, 8:49--128, 1999.

\bibitem{Nagumo}
J.~Nagumo, S.~Arimoto, and S.~Yoshizawa.
\newblock An active pulse transmission line simulating nerve axon.
\newblock {\em Proc. IRE}, 50:2061--2070, 1962.

\bibitem{Pardoux}
E.~Pardoux.
\newblock Equations aux derivees partielles stochastiques non lineaires
  monotones; etude de solutions fortes de type {It\^{o}}, 1975.
\newblock These.

\bibitem{DaPratoZabczyk}
G.~Da Prato and J.~Zabczyk.
\newblock {\em Stochastic Equations in Infinite Dimensions}.
\newblock Cambridge University Press, 1992.

\bibitem{Robinson1}
J.C. Robinson.
\newblock {\em Infinite-Dimensional Dynamical Systems}.
\newblock CUP, 2001.

\bibitem{Sandstede1}
B.~Sandstede.
\newblock Stability of travelling waves.
\newblock In B.~Fiedler, editor, {\em Handbook of Dynamical Systems}, volume~2,
  pages 983--1055. Elsevier, 2001.
	
\bibitem{SauerStannat1}
M.~Sauer and W.~Stannat.
\newblock Lattice approximation for stochastic reaction diffusion equations
  with one-sided {Lipschitz} condition.
\newblock {\em Math. Comp.}, 84:743--766, 2015.

\bibitem{Shardlow}
T.~Shardlow.
\newblock Numerical simulation of stochastic {PDEs} for excitable media.
\newblock {\em J. Comput.. Appl. Math.}, 175(2):429--446, 2005.

\bibitem{Stannat}
W.~Stannat.
\newblock Stability of travelling waves in stochastic {Nagumo} equations.
\newblock {\em arXiv:1301.6378}, pages 1--22, 2013.

\bibitem{Tuckwell1}
H.C. Tuckwell.
\newblock Analytical and simulation results for the stochastic spatial
  {Fitzhugh-Nagumo} model neuron.
\newblock {\em Neural Computation}, 20(12):3003--3033, 2008.

\bibitem{Zeidler2b}
E.~Zeidler.
\newblock {\em Nonlinear Functional Analysis and its Applications {II/B}:
  Nonlinear Monotone Operators}.
\newblock Springer, 2013.

\end{thebibliography}
\end{document}